\documentclass{article}
\usepackage[a4paper, total={6in, 9in}]{geometry}
\usepackage{amsfonts}
\usepackage{amssymb}
\usepackage{amsmath}
\usepackage{amsthm}
\usepackage{mathrsfs}
\usepackage{enumerate}
\usepackage{cite}
\usepackage{color,soul}
\usepackage{setspace}

\usepackage{imakeidx}
\makeindex[name=n,title={Index of notation}]
\makeindex[title={Index of terminology}]

\usepackage[all,cmtip]{xy}
\usepackage{graphicx}

\usepackage{wrapfig}
\usepackage{lipsum}
\usepackage{multirow}

\newtheorem{theorem}{Theorem}[section]
\newtheorem{prop}[theorem]{Proposition}
\newtheorem{lemma}[theorem]{Lemma}

\newtheorem{cor}[theorem]{Corollary}

\theoremstyle{definition}
\newtheorem{definition}[theorem]{Definition}
\newtheorem{remark}[theorem]{Remark}

\newtheorem{hypothesis}[theorem]{Hypothesis}
\newtheorem{notation}[theorem]{Notation}
\newtheorem{example}[theorem]{Example}
\numberwithin{equation}{theorem}
\newcommand{\Z}{\mathbb{Z}}

\newcommand{\I}{\mathbb{I}}
\newcommand{\A}{\mathbb{A}}
\newcommand{\N}{\mathbb{N}}
\newcommand{\bbP}{\mathbb{P}}
\newcommand{\bbO}{\mathbb{O}}
\newcommand{\OX}{\mathcal{O}_X}

\newcommand{\MM}{\mathcal{M}}
\newcommand{\HH}{\mathcal{H}}
\newcommand{\NN}{\mathcal{N}}
\newcommand{\OO}{\mathcal{O}}
\newcommand{\LL}{\mathcal{L}}
\newcommand{\GK}{\mathrm{GK}}
\newcommand{\GKdim}{\mathrm{GKdim}}

\newcommand{\Ext}{\mathrm{Ext}}

\newcommand{\Hom}{\mathrm{Hom}}
\newcommand{\End}{\mathrm{End}}

\newcommand{\coh}{\mathrm{coh}}
\newcommand{\gr}{\mathrm{gr}}
\newcommand{\Gr}{\mathrm{Gr}}
\newcommand{\Gl}{\mathrm{Gl}}
\newcommand{\rann}{\mathrm{r.ann}}

\newcommand{\qgr}{\mathrm{qgr}}

\newcommand{\Sg}{S_{(g)}}
\newcommand{\bfd}{\mathbf{d}}
\newcommand{\bfe}{\mathbf{e}}

\newcommand{\bfx}{\mathbf{x}}
\newcommand{\bfy}{\mathbf{y}}
\newcommand{\mbf}{\mathbf}

\newcommand{\ovl}{\overline}
\newcommand{\ehd}{\overset{\bullet}=}

\title{Maximal Orders in the Sklyanin Algebra}
\author{Dominic Hipwood}
\date{\today}
\begin{document}

\setlength\abovedisplayskip{3pt}
\setlength\belowdisplayskip{3pt}

\maketitle
\begin{abstract}

A major current goal of noncommutative geometry is the classification of noncommutative projective surfaces. The generic case is to understand algebras birational to the Sklyanin algebra. In this work we complete a considerable component of this problem.\par

Let $S$ denote the 3-dimensional Sklyanin algebra over an algebraically closed field, and assume that $S$ is not a finite module over its centre. In earlier work Rogalski, Sierra and Stafford classified the maximal orders inside the 3-Veronese $S^{(3)}$ of $S$. We complete and extend their work and classify all maximal orders inside $S$. As in Rogalski, Sierra and Stafford's work, these can be viewed as blowups at (possibly non-effective) divisors. A consequence of this classification is that maximal orders are automatically noetherian among other desirable properties. \par

\end{abstract}
\tableofcontents

\section{Introduction}

Roughly, the goal of noncommutative projective geometry is to use techniques and intuition from commutative geometry to study noncommutative algebras. A major current goal of noncommutative geometry is the classification of noncommutative projective surfaces; in the language of algebras, a classification of connected graded domains of GK-dimension 3. Within this, a major target is a classification of algebras birational to the Sklyanin algebra. This problem provides the main motivation of our work. \par

\subsection*{First definitions and notations}
First, as always, we need to get some essential definitions and formalities out of the way. 
Fix an algebraically closed field $\Bbbk$. Let $A=\bigoplus_{n\in\N} A_i$  be an $\N$-graded  $\Bbbk$-algebra and a domain. We call $A$ \textit{connected graded (cg)} if $A_0=\Bbbk$ and $\dim_\Bbbk A_i<\infty$ for all $i$. 
Almost all algebras considered will be cg domains of finite Gelfand Kirillov dimension (see Notation~\ref{GK dimension} for a definition). In particular we assume this for the next few definitions.\par

The algebra $A$ sits inside its \textit{graded quotient ring}, $Q=Q_\gr(A)$, formed by inverting the Ore set of all nonzero homogenous elements. The division ring of degree zero elements of $Q$ will be denoted by $D_\gr(A)$. We call $A$ a \textit{maximal order} if there is no cg $B$ with  $A\subsetneq B\subseteq Q$, and such that $xBy\subseteq A$ for some nonzero $x,y\in Q$. Let $R$ be another algebra such that $A\subseteq R\subseteq Q$. We call $A$ a \textit{a maximal $R$-order} if there exists no graded equivalent order $B$ to $A$ such that $A\subsetneq B\subseteq R$. Given $d\in\N$, the \textit{$d$-Veronese of $A$} is the subring $A^{(d)}=\bigoplus_{i\in\N}A_{di}$. It is (usually) given the grading $A^{(d)}_n=A_{dn}$.\par

A construction that plays a fundamental role in our work, and in general noncommutative geometry, is that of a twisted homogenous coordinate ring. Fix a projective scheme $X$. Let $\LL$ be an invertible sheaf on $X$ with global sections $H^0(X,\LL)$, and let $\sigma:X \rightarrow X $ be an automorphism of $X$. Write $\LL^\sigma$ for the pullback sheaf $\sigma^*\LL$. Set $\LL_0=\OX$ and $\LL_n=\LL\otimes \LL^\sigma \otimes \dots \otimes \LL^{\sigma^{n-1}}$ for $n\geq 1$. We define $B(X,\LL,\sigma)=\bigoplus_{n\in\N}H^0(X,\LL_n)$. There is a natural ring structure on $B(X,\LL,\sigma)$ and we call it a \textit{twisted homogenous coordinate ring}.\par

Our results regard certain subalgebras of the Sklyanin algebra, which we now define. Let $a,b,c\in\Bbbk$, then we set
$$S=\frac{\Bbbk\langle x,y,z \rangle}{(azy+byz+cx^2, axz+bzx+cy^2, ayx+bxy+cz^2)}.$$
Provided $a,b,c$ are general enough, then $S$ has a central element $g\in S_3$, unique up to scalar multiplication and such that $S/gS \cong B(E, \LL, \sigma)$. Here $E$ is a nonsingular elliptic curve, $\LL$ is an invertible sheaf of degree 3, and $\sigma$ is an automorphism of $E$. We call such $S$ a \textit{Sklyanin algebra}. \par

\subsection*{The main results}
 The results of this work are analogous to those of \cite{Ro, RSS}, where the authors tackled similar problems inside the 3-Veronese subring $T=S^{(3)}$ of $S$. Here, as in \cite{Ro, RSS}, our results concern certain \textit{blowup subalgebras} and \textit{virtual blowup subalgebras} of $S$ at effective and so-called \textit{virtually effective divisors} on $E$. Our assumption throughout is Hypothesis~\ref{standing assumption intro}.

\begin{hypothesis}[Standing Assumption]\label{standing assumption intro}
Fix an algebraically closed field $\Bbbk$. Fix a 3 dimensional Sklyanin algebra $S$. Let $g\in S_3$ be such that $S/gS\cong B(E,\LL,\sigma)$, where $E$ is a smooth elliptic curve, $\LL$ an invertible sheaf on $E$ with $\deg\LL=3$ and $\sigma:E\to E$ an automorphism of $E$. Assume that $\sigma$ is of infinite order.
\end{hypothesis}

Before presenting our results we set some notation as standard.

\begin{notation}\label{ovlX notation}\index[n]{x@$\ovl{X}$}
Given a subset $X\subseteq S$, we write $\ovl{X}=(X+gS)/gS \subseteq S/gS$. Similarly for $x\in S$, $\ovl{x}$ will denote its image in $\ovl{S}=B(E,\LL,\sigma)$.
\end{notation}

In \cite{Ro} Rogalski studied what will be our 1 point blowup. The definition becomes slightly more complicated as we allow blowups at 2 points, since now the rings need not be generated in a single degree. 

\begin{definition}[Definition~\ref{S(p)} and Definition~\ref{S(p+q)}]\label{S(d) def} \index[n]{sd@$S(\bfd)$}\index{blowup of $S$ at $\bfd$}
Let $\bfd$ be an effective divisor on $E$ with $\deg\bfd\leq 2$. For $i=1,2,3$, put
\begin{itemize}
\item $V_1=\{x\in S_1\,|\; \ovl{x}\in H^0(E,\LL(-\bfd))\}$,
\item $V_2=\{x\in S_2\,|\; \ovl{x}\in H^0(E,\LL_2(-\bfd-\sigma^{-1}(\bfd)))\}$,
\item $V_3=\{x\in S_3\,|\; \ovl{x}\in H^0(E,\LL_3(-\bfd-\sigma^{-1}(\bfd)-\sigma^{-2}(\bfd)))\}$.
\end{itemize}
We define the \textit{blowup of $S$ at $\bfd$} to be the subalgebra of $S$ generated by $V_1$, $V_2$ and $V_3$. We denote it by $S(\bfd)=\Bbbk\langle V_1,V_2,V_3\rangle$.
\end{definition}

When $\bfd=p$ is a single point, we in fact have $S(p)=\Bbbk\langle V_1\rangle$. Rogalski showed in \cite{Ro} that it is the only degree 1 generated maximal order inside $S$. When $\bfd=p+q$ is two points Definition~\ref{S(d) def} is both new, and harder to understand than $S(p)$. The $S(\bfd)$ are the correct analogue of the blowup subalgebras $S^{(3)}(\bfe)$ of $S^{(3)}$ defined by Rogalski in \cite[Section~1]{Ro}. We now state the main results of this work.

\begin{theorem}[Theorem~\ref{S(d) thm}]\label{S(d) thm intro}
Let $\bfd$ be an effective divisor on $E$ of degree $d\leq 2$. Set $R=S(\bfd)$. Then:
\begin{enumerate}[(1)]
\item $R\cap gS=gR$ with $R/gR\cong B(E,\LL(-\bfd),\sigma)$. The Hilbert series of $R$ is
$$h_{R}(t)=\sum_n(\dim_\Bbbk R_n)t^n=\frac{t^2+(1-d)t+1}{(t-1)^2(1-t^3)}.$$
\item The 3-Veronese $R^{(3)}$ is a blowup subalgebra of $S^{(3)}$. More specifically
    $$R^{(3)}=S^{(3)}(\bfd+\sigma^{-1}(\bfd)+\sigma^{-2}(\bfd)).$$
\item $R$ is a maximal order in $Q_\gr(R)=Q_\gr(S)$.
\end{enumerate}
\end{theorem}

We also show in Theorem~\ref{S(d) thm} that $S(\bfd)$ satisfies some of the most useful homological properties, the most prominent for us will the Auslander-Gorenstein and Cohen-Macaulay properties (see Definition~\ref{homological defs}). Although not explicit below, obtaining Theorem~\ref{S(d) thm intro} is absolutely essential for the rest of our main results.\par

For a complete classification of maximal $S$-orders we need to introduce virtual blowups at virtually effective divisors. We only define a virtually effective divisor here. For the purposes of the introduction one may take the conclusions of Proposition~\ref{vblowup exist intro} as the definition of a virtual blowup.

\begin{definition}\label{veff div intro}
A divisor $\bfx$ is called \textit{virtually effective} if for all $n\gg0$, the divisor
$\bfx+\sigma^{-1}(\bfx)+\dots+\sigma^{-(n-1)}(\bfx)$
is effective.
\end{definition}

\begin{prop}[Proposition~\ref{RSS 7.4(3)} and Definition~\ref{virtual blowup}]\label{vblowup exist intro}
Let $\mbf{x}$ be a virtually effective divisor of degree at most 2. Then there exists a virtual blowup $S(\bfx)$. In particular:
\begin{enumerate}[(1)]
\item $S(\bfx)$ is a maximal order in $Q_\gr(S)$ and uniquely defines a maximal $S$-order $V=S(\bfx)\cap S$.
\item $S(\bfx)\cap gS=gS(\bfx)$ and so $S(\bfx)/gS(\bfx)\cong \ovl{S(\bfx)}$. Moreover, in high degrees $n\gg 0$,
    $$\ovl{S(\bfx)}_{\geq n} =B(E,\LL(-\bfx),\sigma)_{\geq n}.$$
\end{enumerate}
\end{prop}

Despite the current notation, it is unknown whether the algebra $S(\bfx)$ appearing in Proposition~\ref{vblowup exist intro} is unique for a fixed virtually effective divisor $\bfx$. Another unknown is when $S(\bfx)\subseteq S$ holds. These problems are investigated in \cite[Chapter~3.5]{thesis}. \par

The classification of maximal $S$-orders remarkably only requires blowups and virtual blowups at at most 2 points.

\begin{theorem}[Theorem~\ref{RSS 8.11}]\label{main result 1}
Let $U$ be a connected graded maximal $S$-order such that $\ovl{U}\neq\Bbbk$. Then there exists a virtually effective divisor $\bfx$ with $0\leq\deg\bfx\leq 2$, and a virtual blowup $S(\bfx)$, such that $S(\bfx)$ is the unique maximal order containing $U=S(\bfx)\cap S$.
\end{theorem}

It may not be the case that $S(\bfx)\subseteq S$, however the difference between $S(\bfx)$ and $U=S(\bfx)\cap S$ is small (see Theorem~\ref{main thm converse}(1c)). Out of proving the above theorems we also obtain many nice properties for maximal orders. The most striking of these is that we get that maximal ($S$-)orders are noetherian for free. The homological terms in Corollary~\ref{blowup properties}(2) will remain undefined for they are not used. The definitions can be found in \cite[Chapter~2]{thesis}.

\begin{cor}[Corollary~\ref{RSS 8.11'} and Corollary~\ref{RSS 8.12}]\label{blowup properties}
Let $U$ be a cg maximal $S$-order such that $\ovl{U}\neq \Bbbk$. Equivalently, let $U=S(\bfx)\cap S$ for some virtual blowup $S(\bfx)$ at a virtually effective divisor with $0\leq\deg \bfx\leq 2$. Then
\begin{enumerate}[(1)]
\item $S(\bfx)$ and $U$ are strongly noetherian, and are finitely generated as $\Bbbk$-algebras.
 \item $S(\bfx)$ and $U$ satisfy the Artin-Zhang $\chi$ conditions, have finite cohomological dimension, and possess balanced dualizing complexes.
\end{enumerate}
\end{cor}

The properties Auslander-Gorenstein and Cohen-Macaulay properties are missing from Corollary~\ref{blowup properties}. It is shown in Theorem~\ref{S(p-p1+p2)}(3) that in general we cannot expect these homological properties. \par

The ultimate goal of this line of research is to understand connected graded algebras birational to the Sklyanin algebra. In other words, algebras $A$ satisfying $D_\gr(A)=D_\gr(S)$; or equivalently $Q_\gr(A)=Q_\gr(S)^{(d)}$ for some $d\geq 1$.



\begin{theorem}[Theorem~\ref{RSS 8.11}]\label{RSS 8.11 intro}
Let $d\geq1$ be coprime to 3 and suppose that $U$ is a cg maximal $S^{(d)}$-order satisfying $\ovl{U}\neq \Bbbk$. Then there exists a virtually effective divisor $\bfx$ with $0\leq\deg \bfx\leq 2$, and virtual blowup $S(\bfx)$, such that $U=S(\bfx)\cap S^{(d)}$.
\end{theorem}

Retain the notation of Theorem~\ref{RSS 8.11 intro}. When $d=3e$ is divisible by 3, $U=(F\cap S)^{(e)}$ for a virtual blowup $F$ of $S^{(3)}$. This is \cite[Theorem~8.11]{RSS}. In contrast, when we prove Theorem~\ref{main thm converse} - the converse of Theorem~\ref{RSS 8.11 intro} - we also prove the analogous statement for maximal $S^{(3)}$-orders. This result is an improvement on \cite{RSS}. Out of these results we are able to obtain the best answer yet to \cite[Question~9.4]{RSS}.

\begin{cor}[Corollary~\ref{g-div max orders up n down generalised}]\label{g-div max orders up n down generalised intro}
Let $U$ be a cg graded subalgebra of $S$ satisfying $D_\gr(U)=D_\gr(S)$ and such that $\ovl{U}\neq \Bbbk$. If $U$ is a maximal order then $U^{(d)}$ is a maximal order for all $d\geq 1$.
\end{cor}
\par
A final achievement of this work is an explicit construction of a virtual blowup - a first of its kind. The example shows that these algebras have certain intriguing and more technical properties, and the reader is referred to Section~6 for details. 


\subsection*{History and motivation}\label{history}

In 1987, Artin and Schelter started a project to classify the noncommutative analogues of polynomial rings in 3 variables \cite{ASc}. These algebras are now called AS-regular algebras. The subject of noncommutative projective geometry was born when Artin, Tate and Van den Bergh completed this classification in \cite{ATV1, ATV2}. Their results can be thought of as a classification of noncommutative projective planes. \par

More generally, one would like to classify all so-called noncommutative curves and surfaces. Let $A$ be a cg noetherian domain, then we can associate a noncommutative projective scheme $\qgr(A)$ to $A$. It can be thought of as the noncommutative analogue of coherent sheaves, $\coh(X)$, over (the non-existent) $X=\mathrm{Proj}(A)$. The classification for noncommutative projective curves, when the GK-dimension is 2, was completed by Artin and Stafford in \cite{AS}. They show that in this case $\qgr(A)\sim\coh(X)$ for a genuine integral projective curve $X$. The question of noncommutative surfaces (when $\GKdim(A)=3$) is still very much open. It is this ultimate goal that motivates this work.\par

Let $A$ be a cg domain with $\GKdim(A)=3$. Then $Q_\gr(A)=D_\gr(A)[t,t^{-1};\alpha]$; a skew Laurent polynomial ring over the division ring $D_\gr(A)$ which is of transcendence degree 2. The division ring $D_\gr(A)$ is often called the \textit{noncommutative function field of $A$}. A programme for the classification is to first classify the possible birational classes (the possible $D_\gr(A)$'s), and then classify the algebras in each birational equivalence class. Artin conjectures in \cite{Ar} that we know all the possible division rings. They are:
\begin{enumerate}[(1)]
\item A division algebra which is finite dimensional over a central commutative subfield of transcendence degree 2.
\item A division ring of factions of a skew polynomial extensions of $\Bbbk(X)$, for a commutative curve $X$.
\item A noncommutative function field $D_{\gr}(S)$ of a 3 dimensional Sklyanin algebra $S$.
\end{enumerate}
Whilst this conjecture is still a long way off, significant work has been, and is being, done on the classification of algebras in each birational class. Algebras with $D_\gr(A)$ commutative (plus a geometric condition) have been successfully classified by Rogalski and Stafford and then Sierra in \cite{RoSt.naive.nc.blowups, RoSt.class.of.nc.surfaces} and \cite{Si} respectively. This is a significant subclass of (1) above.  We are interested in case (3) when $D_\gr(A)=D_\gr(S)$. More specifically, we look at subalgebras $A$ of $S$ with $D_\gr(A)=D_\gr(S)$. Where would be a good place to start looking for such algebras? Inside $S$ of course! How about a target to aim for? Maximal orders are the noncommutative analogue of integrally closed domains, or geometrically, of normal varieties. They are therefore a natural target for such a classification. \par
The first major results in this direction were given by Rogalski. In \cite{Ro} Rogalski successfully classifies the degree 1 generated maximal orders of the 3-Veronese ring $T=S^{(3)}$ of $S$. These are classified as so-called blowup subalgebras $T(\bfd)$ of $T$ at effective divisors on $E$ of degree at most $7$. This is extended to include all maximal orders and maximal $T$-orders by Rogalski, Sierra and Stafford in \cite{RSS, RSS2}. A detailed review of their work can be found in \cite[Chapter~2.6-2.7]{thesis}. In this work we ask the question, why work with $S^{(3)}$? Surely it is $S$ we are interested in?!

\section{$g$-divisible subalgebras of the Sklyanin algebra}\label{prelims}

We fix Hypothesis~\ref{standing assumption intro} and its notation once and for all. What is clear from the work of Rogalski, Sierra and Stafford is that the property of $g$-divisibility (Definition~\ref{g div def}) will play an important role: a $g$-divisible subring is significantly easier to describe as properties pass a lot more smoothly between a ring and its image in $\ovl{S}=B(E,\LL,\sigma)$. Many of the results in this section are in fact proved in the slightly bigger ring obtained by inverting elements outside $gS$. 

\begin{notation}\label{S_(g)}\index[n]{sg@$\Sg$} The set $\mathcal{C}$ consisting of the homogenous elements in $S\setminus gS$ is an Ore set. We will denote $S_{(g)}$ to be homogenous localisation of $S$ at the completely prime ideal $gS$. That is, $S_{(g)}=S\mathcal{C}^{-1}$. It is clear that $Q_\gr(\Sg)=Q_\gr(S)$ and $\label{ovlS_(g)}\Sg/g\Sg=Q_\gr(S/gS)=\Bbbk(E)[t,t^{-1};\sigma]$. For a subset $X\subseteq \Sg$, we extended Notation~\ref{ovlX notation} to $\Sg$: $\ovl{X}=(X+g\Sg)/g\Sg\subseteq \Bbbk(E)[t,t^{-1};\sigma]\index[n]{x@$\ovl{X}$}$. As $S\cap g\Sg=gS$ one has $S/gS\hookrightarrow \Sg/g\Sg$; in particular the notation $\ovl{X}$ is unambiguous.
\end{notation}




\begin{definition}\label{g div def}\index{gdivisible@$g$-divisible}  \index{gdivisiblehull@$g$-divisible hull} \index[n]{mr@$\widehat{M}$, $\widehat{R}$}
Let $R\subseteq S_{(g)}$ be a cg subalgebra containing $g$, and suppose that $M\subseteq S_{(g)}$ is a graded right $R$-module.  We say $M$ is \textit{$g$-divisible} if $M\cap Sg= gM$. The graded right $R$-module $\widehat{M}=\{a\in S_{(g)}|\, ag^n\in M \text{ for some } n\geq 0\} $ is called the \textit{$g$-divisible hull} of $M$. It is easy to show that $M$ is $g$-divisible if and only if $M=\widehat{M}$. The ring $R$ is \textit{$g$-divisible} if it is $g$-divisible as a right (equivalently left) module over itself.
\end{definition}

\begin{remark}\label{RSS section 2}
In \cite[Section~2]{RSS} the authors studied $g$-divisible subalgebras of $S^{(3)}$, or more specifically, $g$-divisible subalgebras of algebras satisfying    \cite[Assumption~2.1]{RSS}. Our $S$ (from Hypothesis~\ref{standing assumption intro}) satisfies \cite[Assumption~2.1]{RSS} in all expect that $\deg(g)=3$ instead of $\deg(g)=1$. However the assumption $\deg(g)=1$ is never used in proving all results \cite[2.8-2.16]{RSS}, and in fact these results and proofs hold under $\deg(g)=3$. Details of this can be found in \cite[Chapter~3.1-3.4]{RSS}. We hence will be using their results for our case on top of those that we prove here. The specific results which we will use in this paper under the our more general assumptions are \cite[Lemmas~2.10, 2.12, 2.13, 2.14 and Propositions 2.9, 2.16]{RSS}.
\end{remark}

Let $A$ be a graded domain with graded quotient ring $Q=Q_\gr(A)$. Let $M,N$ be nonzero right $A$-submodules of $Q$. Then $\Hom_A(M,N)$ is a graded subspace of $Q$, and a simple application of the Ore condition shows that we can identify $\Hom_R(M,N)=\{ q\in Q\,|\; qM\subseteq N\}$. We make this identification a standard throughout. In particular 
$$\End_A(M)=\{q\in Q \, |\; qM\subseteq M\} \;\text{ and }\; M^*=\Hom_A(M,A)=\{q\in Q \, |\; qM\subseteq A\}.$$
We also make the analogous identifications on the left.

\begin{lemma}\label{hat end commute}
Let $R$ be a $g$-divisible cg subalgebra of $\Sg$. Let $M\subseteq \Sg$ be a finitely generated right $R$-module. Then $\End_R(\widehat{M})=\widehat{\End_R(M)}.$
\end{lemma}
\begin{proof}
By \cite[Lemma~2.13(2)]{RSS} and Remark~\ref{RSS section 2}, $\widehat{M}$ is finitely generated, say $\widehat{M}=x_1R+\dots+x_nR$. Each $x_i\in \widehat{M}$, so there exists $k_i\geq 0$ such that $g^{k_i}x_i\in M$. Hence, if $k=\max_i\{k_i\}$, then $g^k\widehat{M}\subseteq M$. Set $U=\End_R(M)$ and $V=\End_R(\widehat{M})$. By \cite[Lemma~2.12(3)]{RSS} and Remark~\ref{RSS section 2}, $V$ is $g$-divisible, and hence to prove $\widehat{U}\subseteq V$, it is enough to prove $U\subseteq V$. Let $u\in U$. Then $(ug^k)\widehat{M}=u(g^k\widehat{M})\subseteq uM\subseteq M\subseteq \widehat{M},$
so $ug^k\in V$. Since $V$ is $g$-divisible $u\in V$. Therefore $U\subseteq V$, and then $\widehat{U}\subseteq V$. Now take $v\in V$. Then $(vg^k)\widehat{M}=g^k(v\widehat{M})\subseteq g^k\widehat{M}\subseteq M,$ so certainly $vg^kM\subseteq M$. Therefore $vg^k\in U$ and $v\in\widehat{U}$.
\end{proof}

To help prove many results in $S$, we will want to utilise the results of \cite{Ro} and \cite{RSS} as much as possible. A clear strategy is therefore to pass information between Veronese rings. We develop some basic machinery for this here. 

\begin{lemma}\label{choose x of appropriate deg}
Let $A$ be a cg domain with $Q=Q_\gr(A)$ and such that $Q_1\neq 0$. Then $A_n\neq 0$ for all $n\gg 0$. \qed
\end{lemma}

\begin{lemma}\label{noeth up n down}
Let $A$ be a cg domain. Fix $d\geq 1$ and write $A'=A^{(d)}$.
\begin{enumerate}[(1)]
\item Then $A$ is noetherian if and only if $A'$ is noetherian.
\item Suppose that $A$ is noetherian and let $M\subseteq Q_\gr(A)$ be a finitely generated right $A$-module. Then as a right $A'$-module
$$M=\bigoplus_{k=0}^{d-1}M^{(k\;\mathrm{mod}\, d)}\index[n]{Mkmodd@$M^{(k\;\mathrm{mod}\, d)}$} $$
where $M^{(k\;\mathrm{mod}\, d)}=\bigoplus_{i}M_{di+k}$ are finitely generated right $A'$-modules. In particular $M^{(d)}$, $M$ and $A$ are all finitely generated as right $A'$-modules.
\end{enumerate}
\end{lemma}
\begin{proof} (1). This  \cite[Proposition~5.10]{AZ.ncps} and \cite[Lemma~4.10]{AS}.\par
(2). Write $Q=Q_\gr(A)$. Fix $k=0,1,2,\dots,d-1$ and write $M'=M^{(k\;\mathrm{mod}\,d)}$. If $M=0$ there is nothing to prove so assume $M\neq 0$. Regrading if necessary, we may assume both $A'\neq 0$ and $M'\neq 0$. Since $M_A$ is finitely generated there exists a homogeneous $x\in A$ such that $xM\subseteq A$. By Lemma~\ref{choose x of appropriate deg} we may multiply $x$ on the left with an element of $A$ with appropriate degree, and hence we may assume $x\in A'$. In which case $xM'\subseteq A\cap Q'=A'$, in particular $M'$ embeds into $A'$ as a right ideal. Since $A$ is noetherian, $A'$ is also noetherian by part (1). Hence $M'_{A'}$ is finitely generated. The module $M$ is then a direct sum of finitely generated modules, hence itself is finitely generated.
\end{proof}

Recall that two cg domains $A$ and $B$ with $Q_\gr(A)=Q_\gr(B)=Q$ are called \textit{equivalent orders} if there exist nonzero elements $x_1,x_2,y_1,y_2\in Q$ such that $x_1Ax_2\subseteq B$ and $y_1By_2\subseteq A$.

\begin{lemma}\label{equiv orders go up}
Let $A$ and $B$ be two cg domains with $Q_\gr(A)=Q_\gr(B)=Q$. Assume that $Q_1\neq 0$ and let $d\geq 1$. If $A$ and $B$ are equivalent orders, then $A^{(d)}$ and $B^{(d)}$ are also equivalent orders with $Q_\gr(A^{(d)})=Q_\gr(B^{(d)})=Q^{(d)}$.
\end{lemma}
\begin{proof}
This is a simple application of Lemma~\ref{choose x of appropriate deg}.
\end{proof}

Since $g\in S_3$, the property of $g$-divisibility also makes sense in the 3-Veronese $T_{(g)}=(\Sg)^{(3)}$. A trivial but useful result is the following.

\begin{lemma}\label{g div up}
Let $R$ be a cg subalgebra of $\Sg$ containing $g$ and let $M$ be a graded right $R$-submodule of $\Sg$. If $M$ is $g$-divisible then $M^{(3)}$ is also $g$-divisible. \qed
\end{lemma}

The next lemma is key to utilising the results of \cite{Ro} and \cite{RSS} fully. It allows us to build appropriate ideals in $S$ out of ideals from $S^{(3)}$ and vice versa. First we must recall the GK dimension of a cg domain. \par

\begin{notation}\label{GK dimension}
Let a $A$ be a cg domain. The \textit{Gelfand-Kirillov (GK) dimension of $A$} can be defined as, and will be denoted by, $\GK(A)=\limsup_{n\geq 0}\log_n(\dim_\Bbbk A_{\leq n} )$, \index{Gelfand-Kirillov (GK) dimension}\index[n]{gk@$\GK(A)$, $\GKdim(A)$}
while $\GK(A)=0$ if and only if $A$ is finite dimensional as a $\Bbbk$-vector space. We will also use the notation $\GKdim(A)$ in sections where GK-dimension is less prominent.  If $M$ is a finitely generated right (or left) $\Z$-graded $A$-module, then similarly
$\GK_A(M)=\limsup_{n\geq 0}\log_n(\dim_\Bbbk M_{\leq n}).$\index[n]{gk@$\GK_A(M)$}
\end{notation}

\begin{lemma}\label{sporadics up n down}
Let $A$ be a cg noetherian domain of finite GK-dimension. Fix $d\in \N$ and write $A'=A^{(d)}$.
\begin{enumerate}[(1)]
\item  For every homogeneous ideal $I$ of $A$, the homogeneous ideal $I'=I^{(d)}$ of $A'$ satisfies $\GK(A'/I')= \GK(A/I)$.
\item Let $J$ be a homogeneous ideal of $A'$.
\begin{enumerate}[(a)]
\item The ideal $I=AJA$ of $A$ satisfies $\GK(A/I)\leq \GK(A'/J)$. Furthermore, $J\subseteq I^{(d)}$ holds.
\item  The ideal $K=\mathrm{r.ann}_A(A/JA)$ of $A$ also satisfies $\GK(A/K)\leq\GK(A'/J)$. In this case $K^{(d)}\subseteq J$.
\end{enumerate}
\end{enumerate}
\end{lemma}
\begin{proof}
In this proof $A'$ is graded as a subset of $A$, i.e. $A'_n=A_n\cap S'$. We also remark that by Lemma~\ref{noeth up n down} we have that $A'$ is noetherian, and that $A$ is finitely generated on both sides as a right and left $A'$-module. Finally, all modules below will always be finitely generated; in particular the GK-dimension of a module is independent of the ring it is being considered over by Notation~\ref{GK dimension}. \par 

(1). One can easily check $A'/I' \cong (A/I)^{(d)}$. Since $A$ is finitely generated on both sides as an $A'$-module, we have that $A/I$ is finitely generated on both sides as an $A'/I'$-module. Thus by \cite[Proposition~5.5]{KL}, $\GK(A'/I')=\GK(A/I)$.\par

(2a). There is an isomorphism of right $A'$-modules $A/JA \cong A'/J \otimes_{A'}A$. Therefore, since $A$ is finitely generated as an $A'$-module on both sides, we can apply \cite[Proposition~5.6]{KL} and deduce $\GK(A/JA)\leq \GK(A'/J)$ as right $A'$-modules.  Clearly $JA\subseteq AJA=I$, and so there is a surjection $A/JA\to A/I$. Hence by \cite[Proposition~5.1(b)]{KL} we have $\GK(A'/J)\geq \GK(A/JA)\geq \GK(A/I)$ as right $A'$-modules. It is obvious that $J\subseteq I^{(d)}$ holds. \par

(2b). To deal with the ideal $K=\mathrm{r.ann}_A(A/JA)$ we first need $\GK(A/JA)=\GK(A'/J).$ We already have $\GK(A/JA)\leq\GK(A'/J)$ from part (2a). For the reverse inequality note that because $(JA)^{(d)}=J$; $A'/J \cong (A/JA)^{(d)}\hookrightarrow A/JA$ as right $A'$-modules. So $\GK(A'/J)\leq \GK(A/JA)$ follows from \cite[Proposition~5.1(b)]{KL}. Hence $\GK(A/JA)\leq\GK(A'/J)$ holds.\par

Now $_{A'}A$ is finitely generated, therefore so is the left $A'$-module $A/JA$. Give generators
$A/JA=A'\ovl{x_1}+\dots+ A'\ovl{x_n}$ where $\ovl{x_i}=x_i+JA$ for some homogenous $x_i\in A$.
It follows that $K=\cap_{i=1}^n K_i$ where $K_i=\rann_A(\ovl{x_i})$ are right ideals of $A$. Fix $i$, and say $\deg(x_i)\equiv \alpha_i\mod d$. By Lemma~\ref{choose x of appropriate deg} we can choose a nonzero $a_i\in A$ such that $\deg(a_i)\equiv -\alpha_i\mod d$. Then $x_ia_iJ\subseteq A'J=J\subseteq JA$ which shows $K_i\neq0$. Because $K$ is an intersection of finitely many nonzero right ideals in a noetherian domain, $K$ must also be nonzero (see \cite[Exercise~4N]{GW}). For every $i$, clearly $(\ovl{x_i}A)_{A}\cong [A/K_i]_{A}$ and $\ovl{x_i}A\subseteq A/JA$, thus $\GK_{A}(A/K_i)\leq \GK(A/JA)=\GK(A'/J)$ by \cite[Proposition~ 5.1(b)]{KL}. But $A/K$ embeds into $A/K_1 \oplus \dots \oplus A/K_n$ in the obvious way, forcing $\GK(A/K)\leq \max_i\{\GK(A/K_i)\}\leq \GK(A'/J).$
Here we are using \cite[Proposition~5.1(a)(b)]{KL}. For the last line, $K^{(d)}= K\cap A'\subseteq JA\cap A'=J.$
\end{proof}

\begin{remark}\label{sporadics up n down applies}
Let $R$ be a cg $g$-divisible subalgebra of $S_{(g)}$ with $Q_\gr(R)=Q_\gr(S)$. By \cite[Proposition~2.9]{RSS} and Remark~\ref{RSS section 2}, $R$ is noetherian. Therefore Lemma~\ref{sporadics up n down} will apply in this situation.
\end{remark}

Our major result of this section partially answers \cite[Question~9.4]{RSS}: we will prove that a $g$-divisible subalgebra $U$ is a maximal order if and only its $d$th Veronese subring $U^{(d)}$ is a maximal order. Our results will be applicable in both $S$ and its 3-Veronese $S^{(3)}=T$.

\begin{notation}\label{3 Veronese notation2}\index[n]{emt@$(E,\MM,\tau)$} \index[n]{t@$T$} \index[n]{tau@$\tau$} \index[n]{tg@$T_{(g)}$}
Let $T=S^{(3)}$, $T_{(g)}=(\Sg)^{(3)}$, $\MM=\LL_3$ and $\tau=\sigma^3$. Whence $T/gT=B(E,\MM,\tau)$ and $T_{(g)}/gT_{(g)}=\Bbbk(E)[t,t^{-1};\tau]$.
\end{notation}



\begin{lemma}\label{Qgr(S) or Qgr(T)}
Let $U\subseteq Q_\gr(S)$ be a subalgebra containing $g$ and such that $D_\gr(U)=D_\gr(S)$. Then either $Q_\gr(U)=Q_\gr(S)$ or $Q_\gr(U)=Q_\gr(T)$.
\end{lemma}
\begin{proof}
Since $g\in U$, certainly $D_\gr(S)[g,g^{-1}]\subseteq Q_\gr(U)\subseteq Q_\gr(S)$. Furthermore, it is standard (see \cite[A.14.3]{NV}) that $D_\gr(S)[g,g^{-1}]=Q_\gr(S)^{(3)}=Q_\gr(T)$. Thus $Q_\gr(T)\subseteq Q_\gr(U)\subseteq Q_\gr(S)$. If $Q_\gr(U)\supsetneq Q_\gr(T)$, then we can pick homogeneous $x\in Q_\gr(U)$, with $\deg(x)$ coprime to 3. In which case there exists $k,\ell\in \N$ such that $x^kg^{-\ell}\in Q_\gr(U)_1$, which forces $Q_\gr(U)=Q_\gr(S)$.
\end{proof}

The next proposition shows maximal $\Sg$-orders are in fact maximal orders. For a cg subalgebra $A$ of $Q_\gr(S)$, we denote the subalgebra of $Q_\gr(S)$ generated by $A$ and $g$ by $A\langle g\rangle=A+Ag+Ag^2+\dots$\index[n]{ag@$A\langle g\rangle$}.

\begin{lemma}\label{max Sg-order is max order}
Let $U$ be a cg subalgebra of $\Sg$ with $D_\gr(U)=D_\gr(S)$. If $U\subseteq A$ are equivalent orders for some cg subalgebra $A$ of $Q_\gr(S)$, then $A\subseteq \Sg$.
\end{lemma}
\begin{proof}
In this proof we give $T$, $T_{(g)}$ and $Q_\gr(T)$ the grading induced from $Q_\gr(S)$; so for example $T_n=T\cap S_n$. \par
For a contradiction assume that $A\not\subseteq \Sg$. Say that $xAy\subseteq U$ for nonzero $x,y\in Q_\gr(S)$. A standard argument shows that we may assume $x,y\in  U$ (in particular $x,y\in\Sg$) and that they are homogeneous. Set $C=A\langle g\rangle$ and $V=U\langle g\rangle$; clearly $V\subseteq C$ whilst $C\not\subseteq \Sg$. By Lemma~\ref{Qgr(S) or Qgr(T)}, there are two cases:
\begin{enumerate}[(1)]
\item $Q_\gr(C)=Q_\gr(V)=Q_\gr(S)$;
\item $Q_\gr(C)=Q_\gr(V)=Q_\gr(T)$.
\end{enumerate}\par
(1). Take $c\in C\setminus \Sg$. Then certainly $c\in Q_\gr(S)$. Now note that $Q_\gr(S)$ equals $\Sg$ localised at the Ore set $\{1,g^{-1},g^{-2},\dots \}$. So, cancelling $g$'s if necessary, we can write $c=sg^{-n}$ for some $s\in\Sg\setminus g\Sg$ and $n\geq 1$. Write $x=x_1g^k$ and $y=y_1g^\ell$ where $x_1,y_1\in \Sg\setminus g\Sg$ and $k,\ell\geq 0$. Choose an integer $m$ such that $nm>k+\ell$ and consider $xc^my$. On the one hand $c^m\in C$, and therefore $xc^my\in xCy\subseteq V\subseteq \Sg$, while on the other $xc^my=x_1s^my_1g^{k+\ell-nm}$. Now $x_1,y_1\notin g\Sg$, while also $s^m\notin g\Sg$ as $g\Sg$ is a completely prime ideal of $\Sg$. Thus all of $x_1^{-1},y_1^{-1},s^{-m}\in \Sg$. But then $g^{k+\ell-nm}=y_1^{-1}s^{-m}x_1^{-1}xc^my \in\Sg.$
Since $k+\ell-nm<0$, this is our desired contradiction. \par
(2) follows in the same way with $T$ in place of $S$.
\end{proof}

The importance of Lemma~\ref{max Sg-order is max order} is that it allows us to concentrate on rings inside $\Sg$ where we are more comfortable.

\begin{lemma}\label{U'<g>}
Let $d\geq 1$ and $U$ be a $g$-divisible cg subalgebra of $\Sg$ such that $D_\gr(U)=D_\gr(S)$.
\begin{enumerate}[(1)]
\item Suppose that $Q_\gr(U)=Q_\gr(S)$ and $d$ is coprime to 3. Set $U'=U^{(d)}$, then $\widehat{U'\langle g\rangle}=U$.
\item Suppose that $Q_\gr(U)=Q_\gr(T)$. Set $U'=U^{(d)}$, then $\widehat{U'\langle g\rangle}=U$.
\end{enumerate}
\end{lemma}

\begin{proof}
(1). Put $V=\widehat{U'\langle g\rangle}$. Since $U$ is $g$-divisible, $g\in U$, and so $V\subseteq U$. Conversely let $u\in U$. Since $d$ is coprime to $3=\deg(g)$, there exists an integer $n\geq 0$ such that $g^nu\in U\cap S^{(d)}=U'$. In which case $u\in \widehat{U'\langle g\rangle}=V$.\par
(2). This is the same as (1) but with the phrase ``$d$ is coprime to $3$" replaced by ``$\deg(g)=1$".
\end{proof}

\begin{prop}\label{g-div max orders up n down}
Let $d\geq 1$ and let $U$ be a $g$-divisible cg subalgebra of $\Sg$ with $D_\gr(U)=D_\gr(S)$.
\begin{enumerate}[(1)]
\item  Suppose that $Q_\gr(U)=Q_\gr(S)$ and $d$ is coprime to 3, and let $U'=U^{(d)}$. Then:
\begin{enumerate}[(a)]
\item $U$ is a maximal order if and only if $U'$ is a maximal order;
\item $U$ is a maximal $S$-order if and only if $U'$ is a maximal $S^{(d)}$-order.
\end{enumerate}
\item Suppose that $Q_\gr(U)=Q_\gr(T)$, and let $U'=U^{(d)}$. Then:
\begin{enumerate}[(a)]
\item $U$ is a maximal order if and only if $U'$ is a maximal order;
\item $U$ is a maximal $T$-order if and only if $U'$ is a maximal $T^{(d)}$-order.
\end{enumerate}
\end{enumerate}
\end{prop}

\begin{proof}
($1a$) ($\Rightarrow$) Suppose that $U$ is a maximal order. Assume that $U'\subseteq A$ for some graded subalgebra $A$ of $Q_\gr(S)$ such that $xAy\subseteq U'$ for some nonzero homogeneous $x,y\in A$. By Lemma~\ref{max Sg-order is max order}  $A\subseteq \Sg$. Set $C=\widehat{A\langle g\rangle}$. Since $xAy\subseteq U'\subseteq U$ and $g\in U$, $x(Ag^i)y=(xAy)g^i\subseteq U$ for all $i\geq 0$. Thus $xA\langle g\rangle y \subseteq U$. Now let $c\in C$, then $cg^n \in A\langle g\rangle$ for some $n\geq 0$, and hence $x(cg^n)y=(xcy)g^n\in U$. Since $x,y\in A\subseteq \Sg$, $xcy\in\Sg$. Hence we get $xcy\in U$ because $U$ is $g$-divisible. This shows $xCy\subseteq U$. On the other hand, because $A\supseteq U'$, $C=\widehat{A\langle g\rangle}\supseteq \widehat{U'\langle g\rangle}=U$ by Lemma~\ref{U'<g>}. By assumption $U$ is a maximal order and hence $C=U$. It then follows $A\subseteq C\cap (\Sg)^{(d)}=U\cap (\Sg)^{(d)}=U'$, and so $A=U'$.\par

($\Leftarrow$) Suppose now that $U'$ is a maximal order, and assume that $U\subseteq A$ are equivalent orders, for some graded subalgebra $A$ of $Q_\gr(S)$. By Lemma~\ref{max Sg-order is max order}, we know that in fact $A\subseteq \Sg$. By Lemma~\ref{equiv orders go up}, $U'\subseteq A'=A^{(d)}$ are equivalent orders, and so by hypothesis $U'=A'$. By Lemma~\ref{U'<g>}, $\widehat{A'\langle g\rangle}=\widehat{U'\langle g\rangle}= U$, and so it is enough to prove $A\subseteq \widehat{A'\langle g\rangle}$. Let $a\in A$. Note that since $U\subseteq A$, $g\in A$. Because $d$ and $3=\deg(g)$ are coprime there exists $n\geq 0$ such that $ag^n\in A\cap S^{(d)}=A'$. Thus $a\in \widehat{A'\langle g\rangle}$. \par

Part ($1b$) is proved in the same way as (1a). This time uses of Lemma~\ref{max Sg-order is max order} are redundant since $A\subseteq S$ by assumption. \par

(2). This follows in the same as part (1) with ``$d$ is coprime to $3$" replaced with ``$\deg(g)=1$".
\end{proof}

\section{The noncommutative blowups $S(\bfd)$}\label{The rings S(d)}

In the classification maximal $S^{(3)}$-orders, other maximal orders are built as endomorphism rings over the blowup subalgebras $T(\bfd)$ as defined in \cite[Section~1]{Ro} (denoted $R(D)$ there). To hope to apply a similar strategy we must ask what are the subalgebras of $S$ are the appropriate analogue of the $T(\bfd)$. This section is dedicated answering exactly this. We will define and study the noncommutative blowups $T(\bfd)$ and their generalisations, $S(\bfd)$, that were defined in introduction. In particular, we will study the 2-point blowup of $S$ in considerable detail. \par
As with the $T(\bfd)$ in \cite[Theorem~1.1]{Ro}, we should expect these $S$-analogues to satisfy many desirable properties. A couple such properties, not yet defined and important for our work, are the Auslander-Gorenstein and Cohen-Macaulay conditions.

\begin{definition}\label{homological defs}
Let $A$ be a ring and $M$ a right $A$-module. The \textit{grade of $M$}\index{grade of a module} is defined as
$$j(M_A)=\inf\{i\,|\;\Ext_A^i(M,A)\neq 0\},$$\index[n]{jm@$j(M)$}
where we allow $j(M_A)=\infty$ if no such $i$ exists. We say $M$ satisfies the \textit{Auslander condition}\index{Auslander condition} if for every $i\geq 0$ and all left submodules $N\subseteq \Ext^i(M,A)$ we have $j(_AN)\geq i$. The grade and the Auslander condition have analogous definition for left $A$-modules. \par
The ring $A$ is \textit{Auslander-Gorenstein}\index{Auslander-Gorenstein} if the modules $A_A$ and $_AA$ have the same finite injective dimension, and every finitely generated left and right $A$-module satisfies the Auslander condition. Suppose that in addition that $\GK(A)\in\Z$. Then we say $A$ is \textit{Cohen-Macaulay}\index{Cohen-Macaulay} if
$$\GKdim(M)+j(M)=\GKdim(A)$$
for all finitely generated left and right modules $M$.
\end{definition}

We now define these $T(\bfd)$ which are so crucial to Rogalski, Sierra and Stafford's work in the classifying the maximal $T$-orders in the 3-Veronese $S^{(3)}=T$ of $S$.

\begin{definition}\label{T(d) def}\cite[Section~1]{Ro}
Let $\bfd$ be an effective divisor on $E$ such that $\deg\bfd\leq 7$. Write $ T(\bfd)=\Bbbk\langle T(\bfd)_1\rangle$ where
$T(\bfd)_1=\{ x\in T_1\,|\; \ovl{x}\in H^0(E,\LL_3(-\bfd))\}$\index[n]{td@$T(\bfd)$}.
The subalgebra $T(\bfd)$ of $T$ is called the \textit{blowup of $T$ at $\bfd$}.\index{blowup of $T$ at $\bfd$}
\end{definition}

The rings that will play a similar role in the classification of maximal $S$-orders are the $S(\bfd)$ which were defined in Definition~\ref{S(d) def}. We will prove that these rings are very closely related to the $T(\bfd)$ and satisfy similar abstract properties. \par

When $\deg\bfd=1$, the ring $S(\bfd)$ was first studied in \cite[Section~12]{Ro}, where the author proves it to be the only degree 1 generated maximal order in $S$. We recall the definition here.

\begin{definition}\label{S(p)}\index[n]{sp@$S(p)$}
Let $p\in E$. We define the \textit{blowup of $S$ at $p$}\index{blowup of $S$ at $p$} as the ring $S(p)=\Bbbk \langle V\rangle$ where $V=\{x\in S_1\,|\; \ovl{x}\in H^0(E,\LL(-p))\}$.
\end{definition}

Rogalski proves strong results on $S(p)$.

\begin{prop}\label{Ro 12.2} \cite[Theorem~12.2]{Ro}
Let $p\in E$ and $R=S(p)$. Then:
\begin{enumerate}[(1)]
\item $R^{(3)}=S^{(3)}(p+p^\sigma+p^{\sigma^2})$, $R$ is $g$-divisible with
 $R/gR\cong \ovl{R}=B(E,\LL(-p),\sigma),$
 and has Hilbert series $h_{R}(t)=\frac{t^2+1}{(1-t)^2(1-t^3)}$.
\item $R$ is strongly noetherian. The ring $R$ satisfies $\chi$ on the left and right, has cohomological dimension 2 and (by \cite[Proposition~2.4]{RSS2}) $R$ has a balanced dualizing complex.
\item $R$ is Auslander-Gorenstein and Cohen-Macaulay.
\item $R$ is a maximal order in $Q_\gr(R)=Q_\gr(S)$.\qed
\end{enumerate}
\end{prop}

Proposition~\ref{Ro 12.2}(1) is enough to show that Definition~\ref{S(p)} coincides the definition of $S(p)$ given in Definition~\ref{S(d) def}. We prefer Definition~\ref{S(p)} as it emphasises that $S(p)$ is generated in degree 1.

\subsection*{The two point blowup $S(p+q)$}

When $\bfd=p+q$ for some $p,q\in E$, understanding $S(\bfd)$ becomes much harder. The main hindrance is that $S(p+q)$ is no longer generated in a single degree. We introduce some notation which will help us along the way.

\begin{notation}\label{p^sigma^j}
Let $\rho:E\to E$ be an automorphism.  Given $p\in E$, or more generally a divisor $\bfx$ of $E$ we will write $p^{\rho^j}=\rho^{-j}(p)$\index[n]{psj@$p^{\sigma^j}$} and $\bfx^{\rho^j}=\rho^{-j}(\bfx)$ \index[n]{xsj@$\bfx^{\sigma^j}$} for $j\in\Z$.
\end{notation}

Notation~\ref{p^sigma^j} will be typically applied with $\rho=\sigma$ or $\rho=\sigma^3=\tau$. In contrast Notation~\ref{[d]_n} below is reserved for the automorphism $\sigma$ unless explicitly said otherwise.

\begin{notation}\label{[d]_n}
For $n\geq 1$, we put $[\bfx]_n=\bfx+\bfx^\sigma+\dots+\bfx^{\sigma^{n-1}}$, while we define $[\bfx]_0=0$. \index[n]{xn@$[\bfx]_n$}
\end{notation}

Notation~\ref{p^sigma^j} and Notation~\ref{[d]_n}, along with Notation~\ref{ovlX notation} from the introduction, will be applicable for the rest of this paper. Using Notation~\ref{p^sigma^j} and Notation~\ref{[d]_n} we recall the definition of $S(p+q)$.

\begin{definition}\label{S(p+q)}
Let $p,q\in E$. Put $V_i=\{x\in S_i\,|\; \ovl{x}\in H^0(E,\LL_i(-[p+q]_i))\}$ for $i=1,2,3$. We define the \textit{blowup of $S$ at $p+q$}\index{blowup of $S$ at $p+q$} to be the ring $S(p+q)=\Bbbk\langle V_1, V_2, V_3\rangle.$\index[n]{spq@$S(p+q)$}
\end{definition}

Let $p,q\in E$ and $S(p+q)=\Bbbk\langle V_1,V_2,V_3\rangle$ as in Definition~\ref{S(p+q)}. Clearly $\ovl{V_1}=H^0(E,\LL(-p-q))$, and because $gS\subseteq S_{\geq 3}$ one then sees $\dim_\Bbbk V_1=\dim_\Bbbk \ovl{V_1}$. Hence by the Riemann-Roch Theorem, $\dim_\Bbbk V_1=1$. In which case $\Bbbk\langle V_1\rangle \cong \Bbbk[t]$. In particular $S(p+q)$ cannot be generated in degree~1. In \cite[Example~7.8]{thesis} it is shown that, at least for some choices of $p$ and $q$, $S(p+q)$ is not generated in degrees 1 and 2. For general $p$ and $q$, it is unknown if $S(p+q)$ can be generated in degrees 1 and 2. We include $V_3$ in Definition~\ref{S(p+q)} to be sure that $g\in S(p+q)$. It is implicit in our proof of Theorem~\ref{S(p+q) g div} that $S(p+q)=\Bbbk \langle V_1,V_2,g\rangle$.\par  

Let $p,q\in E$, and $V_1$, $V_2$ and $V_3$ be as in Definition~\ref{S(p+q)}. Set $\NN=\LL(-p-q)$. One has $$\NN_2=\NN\otimes \NN^\sigma=\LL(-p-q)\otimes \LL^\sigma(-p^\sigma-q^\sigma)=\LL_2(-p-q-p^\sigma-q^\sigma)=\LL_2(-[p+q]_2).$$
Similarly $\NN_3=\LL_3(-[p+q]_3)$, and in general
\begin{equation}\label{NN_n=LL_n(-[d]_n)} \NN_n=\LL_n(-[p+q]_n)\;\text{ for all }n\geq 1.\end{equation}
 It is then clear from Definition~\ref{S(p+q)} that
\begin{equation}\label{ovlVi=}
\ovl{V_i}=H^0(E,\NN_i)=B(E,\NN,\sigma)_i\;\text{ for } i=1,2,3.
\end{equation}
\par An immediate consequence of Definition~\ref{S(p+q)} and the above is that the image of $S(p+q)$ inside $S/gS=B(E,\LL,\sigma)$ is easy to understand.

\begin{lemma}\label{S(p+q) bar}
Let $p,q\in E$ and $R=S(p+q)$. Then $\ovl{R}=B(E,\NN,\sigma)\subseteq \ovl{S}$ where $\NN=\LL(-p-q)$. The  Hilbert series of $\ovl{R}$ is given by
$$h_{\ovl{R}}(t)=\frac{t^2-t+1}{(t-1)^2}.$$
\end{lemma}
\begin{proof}
Set $B=B(E,\NN,\sigma)$. By (\ref{NN_n=LL_n(-[d]_n)}) we have that $\NN_n\subseteq \LL_n$ for all $n\geq 0$; as a consequences $H^0(E,\NN_n)\subseteq H^0(E,\LL_n)$ for all $n\geq 1$ also. This shows that $B\subseteq B(E,\LL,\sigma)=\ovl{S}$. By definition, $\ovl{R}$ is the subring of $\ovl{S}$ generated by $\ovl{V_1}$, $\ovl{V_2}$ and $\ovl{V_3}$; while from (\ref{ovlVi=}) we have that $\ovl{V_1}= B_1$, $\ovl{V_2}=B_2$ and $\ovl{V_3}=B_3$. By \cite[Lemma~3.1]{Ro}, $B$ is generated in degrees 1 and 2. Thus $\ovl{R}=B$. \par
For the Hilbert series, we note that $\ovl{R}_n=H^0(E,\NN_n)$ for $n\geq 0$. So by the Riemann-Roch Theorem, $\dim_\Bbbk\ovl{R}_n=\deg\NN_n=n$ for $n\geq 1$, while $\dim_\Bbbk \ovl{R}_0=1$.
\end{proof}

As sort of a disclaimer, we warn that future uses of Lemma~\ref{S(p+q) bar} may come without reference. The next result we state is the key result which allows us to understand $S(p+q)$. The majority of this section will be dedicated to proving it.

\begin{theorem}\label{S(p+q) g div}
Let $p,q\in E$. Then
\begin{enumerate}[(1)]
\item $S(p+q)$ is $g$-divisible with $S(p+q)/gS(p+q)\cong \ovl{S(p+q)}=B(E,\LL(-p-q),\sigma)$.
\item The Hilbert series of $S(p+q)$ is given by  $h_{S(p+q)}(t)=\frac{t^2-t+1}{(t-1)^2(1-t^3)}$.
\end{enumerate}
\end{theorem}

The statement and proof of Theorem~\ref{S(p+q) g div} should be compared with those of \cite[Theorem~5.2]{Ro}. Here we are made to work harder to compensate for $S(p+q)$ not being generated in degree one. 

\begin{lemma}\label{1 iff 2 of S(p+q) g div}
Parts (1) and (2) of Theorem~\ref{S(p+q) g div} are equivalent. \qed
\end{lemma}
\begin{proof}
Rogalski proves a version of the Lemma~\ref{1 iff 2 of S(p+q) g div} in the first paragraph of the proof of \cite[Theorem~5.2]{Ro}. The same proof holds here.
\end{proof}

\begin{lemma}\label{S(p+q)_i=V_i}
Let $p,q\in E$ and let $V_1$, $V_2$ and $V_3$ be as in Definition~\ref{S(p+q)}. Then
\begin{enumerate}[(1)]
\item $\dim_\Bbbk V_1=1$, $\dim_\Bbbk V_2=2$ and $\dim_\Bbbk V_3=4$;
\item $S(p+q)_2=V_2$ and $S(p+q)_3=V_3$.
\end{enumerate}
\end{lemma}
\begin{proof}
(1). By the Riemann-Roch Theorem, $\dim_\Bbbk \ovl{V_i}=i$. Since $gS\cap S_{\leq 3} = g\Bbbk$, we have that $\dim_\Bbbk V_i=\dim_\Bbbk\ovl{V_i}=i$ for $i=1,2$, and that $\dim_\Bbbk V_3=\dim_\Bbbk\ovl{V_3}+1=4.$\par
(2). Here we must show $V_1^2\subseteq V_2$ and $V_1^3+V_1V_2+V_2V_1\subseteq V_3$. By Lemma~\ref{S(p+q) bar} $\ovl{V_1^2}\subseteq \ovl{V_2}$. Since $gS\subseteq S_{\geq 3}$, this implies $V_1^2\subseteq V_2$. Set $W=V_1^3+V_1V_2+V_2V_1$. Again by Lemma~\ref{S(p+q) bar}, $\ovl{W}\subseteq H^0(E,\LL_3(-[p+q]_3))$. Then $W\subseteq  \{x\in S_3\,|\;\ovl{x}\in H^0(E,\LL(-[p+q]_3))\}=V_3.$
\end{proof}

A fundamental part of Rogalski’s proof of \cite[Theorem~5.2]{Ro} is a semi-continuity argument that unfortunately does not work for the proof of Theorem~\ref{S(p+q) g div}; instead we use a different approach involving Grassmannians. We write $\Gr(m,V)$\index[n]{grmv@$\Gr(m,V)$}\index{Grassmannian} for the Grassmannian of $m$-dimensional subspaces of a fixed $n$-dimensional $\Bbbk$-vector space $V$. The result required is Lemma~\ref{Grassmannian}; it can be proved by elementary geometric arguments (the reader can also find a detailed argument in the appendix to \cite{thesis}).

\begin{lemma}\label{Grassmannian}
Let $R_2$ and $R_3$ be a $4$-dimensional and $7$-dimensional $\Bbbk$-vector space respectively, and set
$$\Omega_2=\{(W_1,W_2)\in \Gr(3,R_2)^2\,|\; \dim_\Bbbk (W_1\cap W_2)=2\}$$
and
$$
\Omega_3=\{(W_1,W_2,W_3)\in \Gr(6,R_3)^3\,| \dim_\Bbbk (W_1\cap W_2\cap W_3)=4\}.
$$
Then $\Omega_2$ and $\Omega_3$ are Zariski open subsets of $ \Gr(3,R_2)^2$ and $\Gr(6,R_3)^3$ respectively; and the maps $\psi_2:\Omega_2\to \Gr(2,R_2)$ and  $\psi_3:\Omega_3\to \Gr(4,R_3)$, given by
$$\psi_2:(W_1,W_2)\longmapsto W_1\cap W_2 \;\text{ and }\; \psi_3:(W_1,W_2,W_3)\longmapsto W_1\cap W_2\cap W_3$$
are morphisms of varieties. \qed
\end{lemma}

Another concept we require to prove Theorem~\ref{S(p+q) g div} is the geometric quotient of a variety $X$ by a group $G$. Let $G$ be an algebraic group and suppose that $G$ acts upon a variety $X$. If this action $G\times X\to X$ is a morphism of algebraic varieties, then $X$ is called a \textit{$G$-variety} \index{gvar@$G$-variety} \cite[Definition~21.4.1]{TY}. Let $X$ be a $G$-variety; one may then define the \textit{geometric quotient of $X$ by $G$} \cite[Definition~25.3.1]{TY}. The reader is referred to \cite{TY} for the exact definition. Suppose that we have a surjective morphism of varieties $\pi:X\to Y$ such that the fibers of $\pi$ are precisely the $G$-orbits of $X$. Then it is the case that $Y$ is the geometric quotient of $X$ by $G$ provided $Y$ is a normal variety \cite[Proposition~25.34]{TY}.\par

Finally, at a certain point in the proof of Theorem~\ref{S(p+q) g div} we will need the field we are working over to be uncountable. Clearly if $\mathbb{K}\supseteq \Bbbk$ is a field extension, then $\mathbb{K}$ is trivially flat as a $\Bbbk$-module. The following lemma is then immediate.

\begin{lemma}\label{extend field}
Let $\mathbb{K}\supseteq \Bbbk$ be a field extension. Let $V$ a finite dimensional $\Bbbk$-vector space and denote $V_\mathbb{K}=V\otimes_\Bbbk\mathbb{K}$. Then $\dim_\Bbbk V=\dim_\mathbb{K}V_\mathbb{K}$. In particular, for a cg $\Bbbk$-algebra $A$ and $A_\mathbb{K}=A\otimes_\Bbbk\mathbb{K}$, we have the equality of Hilbert series $\sum_{i\in\N} \dim_\Bbbk(A_i)t^i =\sum_{i\in\N}\dim_\mathbb{K}(A_i\otimes_\Bbbk \mathbb{K})t^i$.\qed
\end{lemma}


We are now ready to prove $S(p+q)$ is a $g$-divisible ring.

\begin{proof}[Proof of Theorem~\ref{S(p+q) g div}]
By Lemma~\ref{1 iff 2 of S(p+q) g div}, we can prove both (1) and (2) in tandem. We break the proof up into 3 steps.\par\smallskip

\textit{Step 1}. For now, fix $p,q\in E$ and let $R=S(p+q)$. For Step 1 we prove that $R$ is $g$-divisible with the additional assumption that $p$ and $q$ are in distinct $\sigma$-orbits. We need to assume that $\Bbbk$ is uncountable to be sure such a situation exists; by Lemma~\ref{extend field} we lose no generality in doing so.\par
Put $R'=S(p)\cap S(q)$. Since $S$ is a domain and, $S(p)$ and $S(q)$ are already $g$-divisible from Proposition~\ref{Ro 12.2}, $R'$ is $g$-divisible. We claim that in fact $R'=R$. \par
To prove this fix $i=1,2,3$. Clearly $H^0(E,\LL_i(-[p+q]_i))\subseteq H^0(E,\LL_i(-[p]_i))$, and hence
$$R_i=\{x\in S_i\,|\; \ovl{x}\in H^0(E,\LL_i(-[p+q]_i))\}\subseteq
\{x\in S_i\,|\; \ovl{x}\in H^0(E,\LL(-p)) \}=S(p)_i.$$
Since $R$ is generated by $R_1,R_2,R_3$ we have $R\subseteq S(p)$. Symmetrically we get $R\subseteq S(q)$ also, which forces $R\subseteq R'.$  \par
Fix $n\in\N$. By Lemma~\ref{S(p+q) bar},  $\ovl{R_n}=H^0(E,\LL_n(-[p+q]_n)$. Using the assumption that $p$ and $q$ are in distinct $\sigma$ orbits it follows that $H^0(E,\LL_n(-[p+q]_n))=H^0(E,\LL_n(-[p]_n))\cap H^0(E,\LL_n(-[q]_n))$. Thus $\ovl{R}\subseteq \ovl{R'}\subseteq \ovl{S(p)}\cap\ovl{S(q)}=\ovl{R}$, i.e. $\ovl{R}=\ovl{R'}$. If $n=0,1,2$, then $R_n=R'_n$ follows from  $\ovl{R_n}=\ovl{R'_n}$, so assume that $n\geq 3$. Then by induction we have that $R_n\cap gS\supseteq gR_{n-3}=gR'_{n-3}=R'_{n}\cap gS\supseteq R_n\cap gS.$
This certainly implies $\dim_\Bbbk(R_n\cap gS)=\dim_\Bbbk(R'_n\cap gS)$.  We know $\dim_\Bbbk \ovl{R_n}=\dim_\Bbbk\ovl{R'_n}$ already, and it follows that $R$ and $R'$ have the same Hilbert series. Since $R\subseteq R'$, this forces $R_n=R'_n$. This completes the proof under the additional assumption that $p$ and $q$ are in distinct $\sigma$-orbits.\smallskip\par

\textit{Step 2}. Now fix $p\in E$ and allow $q\in E$ to vary. To get rid of the additional assumption of Step~1 the strategy is to prove that the map $\nu_n: E\to \N, \;\;q\mapsto \dim_\Bbbk S(p+q)_n$ is lower semi-continuous for each $n\geq 0$. In other words, for every $\ell\in\N$, the set $\{q\in E\,|\; \nu_n(q)\leq \ell\}$ is Zariski-closed in $E$. Fix $n\geq 0$. We prove $\nu_n$ is lower semi-continuous in two further sub-steps.
\begin{enumerate}[(2a)]
\item By Lemma~\ref{S(p+q)_i=V_i} there is a map $\theta:E\to \Gr(1,S(p)_1)\times \Gr(2,S(p)_2)\times \Gr(4,S(p)_3)$, given by $\theta: q\mapsto (S(p+q)_1,S(p+q)_2,S(p+q)_3)$. We claim $\theta$ is a morphism of varieties.
\item For each $n\in\N$ consider $\mu_n: \Gr(1,S(p)_1)\times \Gr(2,S(p)_2)\times \Gr(4,S(p)_3)\to \N$, defined as $\mu_n: (Y_1,Y_2,Y_3)\mapsto\dim_\Bbbk(\Bbbk\langle Y_1,Y_2,Y_3\rangle_n)$. We claim $\mu_n$ lower semi-continuous for every $n\in\N$.
\end{enumerate}

Assume for the moment that these claims are true. Since clearly $\nu_n=\mu_n\circ \theta$,
$$\{q\in E\,|\; \nu_n(q)\leq \ell\}=\{q\in E\,|\; \mu_n(\theta(q))\leq \ell\}=\theta^{-1}(C),$$
where $C=\{(Y_1,Y_2,Y_3)\,|\;\mu_n(Y_1,Y_2,Y_3)\leq \ell\}$. If (2b) holds, then $C$ is closed. If also (2a) holds, then $\theta^{-1}(C)$ is closed because then $\theta$ is continuous with respect to the Zariski topology. This would prove $\nu_n$ is lower-semi-continuous completing Step 2.\par\smallskip

\textit{Proof of 2a}. Here, when we say ``morphism" we mean ``morphism of varieties".  Write the map $\theta=(\theta_1,\theta_2,\theta_3),\text{ where }\theta_i: q\mapsto S(p+q)_i$. To contain notation we only prove $\theta_3:E\to\Gr(4,S(p)_3)$ is a morphism. The proof for $\theta_1, \theta_2$ are similar.  We will construct $\theta_3$ as a composite of other morphisms. The first of these is the map $\gamma: E\to E^3,\; q\mapsto (q,q^\sigma,q^{\sigma^2})$, this is clearly are morphism since $\sigma$ is an automorphism. Let $\NN=\LL_3(-[p]_3)$, $W=H^0(E,\NN)=\ovl{S(p)_3}$ (which has $\dim_\Bbbk W=6$) and identify $\Gr(5,W)$ with lines through the dual vector space (i.e. with $\bbP(W^*)$). Next, we have the map $\ovl{\phi}:E\to \Gr(5,W)$ determined by a basis of sections for $\NN$ (see \cite[Theorem~II.7.1]{Ha}). Explicitly this is $\ovl{\phi}:q\mapsto H^0(E,\NN(-q))$. We then extend $\ovl{\phi}$ to obtain a morphism $\phi:E\to\Gr(6,S(p)_3)$ by are post-composing $\ovl{\phi}$ with the morphism $\Gr(5,W)\to \Gr(6,S(p)_3): \;W'\mapsto \{x\in S(p)_3\,|\; \ovl{x}\in W'\}$. Put $\Phi=(\phi,\phi,\phi)$ which maps $E\times E\times E\to \Gr(6,S(p)_3)\times \Gr(6,S(p)_3)\times \Gr(6,S(p)_3)$. We now use Lemma~\ref{Grassmannian} with $R_3=S(p)_3$, and $\Omega_3=\Omega$ and $\psi_3=\psi$  defined as in the lemma. We get the morphism $\psi:\Omega\to \Gr(4,S(p)_3)$. We claim $\theta_3=\psi\circ\Phi\circ \gamma$. \par 
 Given $q\in E$, we have $\Phi(\gamma(q))=(X_{0},X_{1},X_{2})$, where $X_{i}=\{x\in S(p)_3\, |\; \ovl{x}\in H^0(E,\NN_i(-q^{\sigma^{i}}))\}$. Since $\sigma$ has infinite order there are no points of finite order. In particular $p$, $p^\sigma$ and $p^{\sigma^2}$ are all pairwise distinct. It follows that $\ovl{X_{0}}\cap\ovl{X_{1}}\cap \ovl{X_{2}}=H^0(E,\LL_3(-[p+q]_3))=\ovl{S(p+q)_3}$. Certainly $\ovl{X_{0}\cap X_{1}\cap X_{2}}\subseteq \ovl{X_{0}}\cap\ovl{X_{1}}\cap \ovl{X_{2}}$. If $\ovl{y}\in \ovl{X_{0}}\cap\ovl{X_{1}}\cap \ovl{X_{2}}$ for some $y\in S_3$, then there exists $x_i\in X_i$ and $\lambda_i\in S_0=\Bbbk$, such that $y=x_i+g\lambda_i$ for $i=1,2,3$. But by the definition of each $X_i$, $g\lambda_i\in X_i$. Hence $y\in X_{0}\cap X_{1}\cap X_{2}$. This shows $\ovl{X_{0}\cap X_{1}\cap X_{2}}= \ovl{X_{0}}\cap\ovl{X_{1}}\cap \ovl{X_{2}}$. It then follows that $X_0\cap X_1\cap X_2\subseteq \{x\in S_3\,|\;\ovl{x}\in H^0(E,\LL_3(-[p+q]_3))\}=V_3$.
Since $gS_0=g\Bbbk\subseteq X_0\cap X_1\cap X_2$, we in fact have $X_0\cap X_1\cap X_2=V_3$. By Lemma~\ref{S(p+q)_i=V_i}, $V_3=S(p+q)_3$ and $\dim_\Bbbk S(p+q)_3=4$. In particular, $\Phi(\gamma(q))\in \Omega$ and $\psi(\Phi(\gamma(q)))$ is well-defined. Moreover $\psi(\Phi(\gamma(q)))=S(p+q)_3=\theta_3(q).$
Thus $\theta_3$, as a composition of morphisms, is itself a morphism of varieties. \smallskip\par

\textit{Proof of 2b}. Fix bases for $S(p)_1$, $S(p)_2$, and $S(p)_3$. By Proposition~\ref{Ro 12.2}(1), $S(p)_1$, $S(p)_2$ and $S(p)_3$ are respectively 2, 4 and 7 dimensional $\Bbbk$-vector spaces. Identify a vector of $S(p)_1\times S(p)_2\times S(p)_3$ with its coordinates in $\A^2\times\A^4\times \A^7$. Let $\mathcal{U}$ be the collection of $(u,(v_1,v_2),(w_1,\dots,w_4))\in \A^2\times(\A^4)^2\times (\A^7)^4$ such that $u\neq 0$ and both $(v_1,v_2)$ and $(w_1,\dots,w_4)$ are linearly independent collections of vectors in $S(p)_2$ and $S(p)_3$ respectively. Linear independence is an open condition, and so $\mathcal{U}$ is open subset. To ease notation put $\mathcal{H}=\Gr(1,S(p)_1)\times \Gr(2,S(p)_2)\times \Gr(4,S(p)_3)$. We have a natural surjection  
$$\pi: \mathcal{U}\to \mathcal{H}, \;(u,(v_1,v_2),(w_1,\dots,w_4))\mapsto (\mathrm{span}\{u\},\mathrm{span}\{v_1,v_2\},\mathrm{span}\{w_1,\dots,w_4\}).$$
We will show this is nothing but the geometric quotient (in the sense of \cite[Definition~25.3.1]{TY}) of $\mathcal{U}$ by the natural action of $G=\Bbbk^\times\times \Gl_2(\Bbbk)\times \Gl_4(\Bbbk)$
given by left multiplication. To this end, we remark that $\mathcal{U}$ is obviously a irreducible $G$-variety (see \cite[Definition~21.4.1]{TY}) and moreover, given $\mathfrak{h}\in \mathcal{H}$, $\pi^{-1}(\mathfrak{h})$ is a $G$-orbit. By \cite[Proposition~25.3.5]{TY}, it is the case that $\mathcal{H}$ is the desired quotient provided $\mathcal{H}$ is a normal variety. But this is obviously true as $\mathcal{H}$ is a product of Grassmannians, and hence has an open affine cover (see \cite[Lemma~11.15]{Hassett}).\par

Fix $n\in \N$. We now prove $\mu=\mu_n$ is lower semi-continuous. First we lift the map $\mu$ to a map $\mu':\mathcal{U}\to \N$. Given $\mathfrak{u}=(u,(v_1,v_2),(w_1,\dots,w_4))\in \mathcal{U}$ set $$\begin{array}{l}
Y_1(\mathfrak{u})=\mathrm{span}\{u\}\subseteq S(p)_1,\\
Y_2(\mathfrak{u})=\mathrm{span}\{v_1,v_2\}\subseteq S(p)_2, \\
Y_3(\mathfrak{u})=\mathrm{span}\{w_1,w_2,w_3, w_4\}\subseteq S(p)_3,
\end{array}
$$
and define $\mu'(\mathfrak{u})=\dim_\Bbbk(\Bbbk\langle Y_1(\mathfrak{u}), Y_2(\mathfrak{u}), Y_3(\mathfrak{u}) \rangle_n).$
We have the commuting diagram
\[
\xymatrix{
\mathcal{U} \ar[d]_{\pi} \ar[dr]^{\mu'} \\
\mathcal{H} \ar[r]^{\mu} & \N. }
\]
To prove $\mu$ is lower semi-continuous we first show that $\mu'$ is lower semi-continuous.  \par
Let $x_1,\dots, x_m$ be the spanning set of $S(p)_n$ consisting of all degree $n$ products of the original fixed bases of $S(p)_1$, $S(p)_2$ and $S(p)_3$. Take $\mathfrak{u}\in \mathcal{U}$ and write $Y_i=Y_i(\mathfrak{u})$ for $i=1,2,3$. Then $\Bbbk\langle Y_1,Y_2,Y_3\rangle_n$ is spanned by elements of the form $y_i=\sum f_{ij}x_j$, where $f_{ij}$ are polynomials in the coordinates of $\mathfrak{u}$. Take $\ell\in \N$, then $\dim_\Bbbk (\Bbbk\langle Y_1,Y_2,Y_3\rangle_n)\leq \ell$ if and only if no collection of $\ell+1$ of the $y_i$ are linearly independent. This in turn is equivalent to the determinants of all $(\ell+1)\times(\ell+1)$ matrix minors of the matrix $\left(f_{ij}\right)_{i,j}$ vanishing. Thus $\{\mathfrak{u}\in \mathcal{U}\,|\; \mu'(\mathfrak{u})\leq \ell \}$ is closed, proving $\mu'$ is lower semi-continuous. Now set $Z=\{ (Y_1,Y_2,Y_3)\in \mathcal{H}\,|\; \mu(Y_1,Y_2,Y_3)\leq \ell \}$, we must show $Z$ is a closed subset of $\mathcal{H}$. We know $\pi^{-1}(Z)=\{\mathfrak{u}\in \mathcal{U}\,|\; \mu'(\mathfrak{u})\leq \ell \}$ is a closed subset of $\mathcal{U}$ - we have just proved it. Because $\mathcal{H}$ is the geometric quotient of $\mathcal{U}$ by $G$, it indeed follows that $Z=\pi(\pi^{-1}(Z))$ is closed (see \cite[Lemma~25.3.2]{TY}). Thus $\mu$ is lower semi-continuous as claimed.
 \smallskip\par

\textit{Step 3}. Finally we complete the proof. By Lemma~\ref{extend field} we may assume that $\Bbbk$ is uncountable. In which case $E$ has uncountably many points. By Step~1, $S(p+q)$ has the correct Hilbert series when $p$ and $q$ are in different $\sigma$-orbits, say $\dim_\Bbbk S(p+q)_n=\alpha_n$ for $p,q$ in different $\sigma$-orbits. By Step~2, $E_n=\{q\in E\,|\; \dim_\Bbbk S(p+q)_n\leq \alpha_n\}$ is a closed subset of $E$. If $E_n\subsetneq E$, then because $E_n$ is closed, $E_n$ would be finite. But $E_n$ contains the uncountable set $\{q\in E\,|\; q\neq p^{\sigma^i}\,\text{ for all }\, i\in\Z\}$ by Step~1. Hence $E_n=E$. In other words, $\dim_\Bbbk S(p+q)_n\leq \alpha_n$ for all $n\geq 0$ and $p,q\in E$. Thus $h_{S(p+q)}(t)\leq \sum\alpha_nt^n=\frac{t^2-t+1}{(t-1)^2(t^3-1)}$, for any $p,q\in E$. The reverse inequality follows from the fact that $R$ is a domain with $Rg\subseteq R\cap Sg$. By Lemma~\ref{1 iff 2 of S(p+q) g div}, this completes the proof.
\end{proof}


Now we know that the $S(p+q)$'s are $g$-divisible we can quickly obtain the main results of this section. In the following the Veronese subring $T=S^{(3)}$ of $S$ is given the grading induced from $S$, that is $T_n=T\cap S_n$.

\begin{cor}\label{S(p+q) 3 Veronese}
Let $p,q\in E$. Then $S(p+q)^{(3)}=T([p+q]_3)$.
\end{cor}
\begin{proof}
Set $R=S(p+q)$, $R'=T([p+q]_3)$ and $\NN=\LL(-p-q)$. Write $R=\Bbbk\langle V_1,V_2,V_3\rangle$ where $V_i=\{ x\in  S_i\,|\; \ovl{x}\in H^0(E,\NN_i)\}$ as in Definition~\ref{S(p+q)}. By Definition~\ref{T(d) def},  $R'=\Bbbk\langle V_3\rangle$, and so we clearly have $R'\subseteq R^{(3)}$. Since we already have this inclusion, to prove equality it is enough to prove that we have an equality of Hilbert series. By Lemma~\ref{S(p+q) bar} and \cite[Theorem~1.1(1)]{Ro}, $\ovl{R_n}=H^0(E,\NN_n)=\ovl{R'_n}$ for each $n\geq 0$ divisible by 3. Thus, for $\dim_\Bbbk R_n=\dim_\Bbbk R'_n$ to hold, it suffices to prove $R_n\cap Sg= R'_n\cap Sg$. But $R_n\cap Sg=R_{n-3}g$  and  $R'_n\cap Sg=R'_{n-3}g$ by Theorem~\ref{S(p+q) g div} and \cite[Theorem~5.2]{Ro}. By induction it is therefore enough to prove $\dim_\Bbbk R_0=\dim_\Bbbk R'_0$, which is obvious.
\end{proof}

\begin{cor}\label{Qgr(S(p+q))}
Let $p,q\in E$ and $R=S(p+q)$. Then $Q_\gr(R)=Q_\gr(S)$.
\end{cor}
\begin{proof}
Trivially $Q_\gr(R)\subseteq Q_\gr(S)$. By Corollary~\ref{S(p+q) 3 Veronese}, $R^{(3)}=T([p+q]_3)$, for which we know that $Q_\gr(T([p+q]_3))=Q_\gr(T)$ by \cite[Theorem~5.4(2)]{Ro}. Thus $D_\gr(R)= D_\gr(T)=D_\gr(S)$. By Lemma~\ref{Qgr(S) or Qgr(T)} either $Q_\gr(R)=Q_\gr(T)$ or $Q_\gr(R)=Q_\gr(S)$. Since $R_1\neq 0$, it is must be that $Q_\gr(R)=Q_\gr(S)$.
\end{proof}

We now present our main theorem of this section. For future referencing purposes we incorporate Proposition~\ref{Ro 12.2} into Theorem~\ref{S(d) thm}. The homological terms in Theorem~\ref{S(d) thm}(3) will remain undefined here for they are not used. The definitions can be found in \cite[Chapter~2]{thesis}. Auslander-Gorenstein and Cohen-Macaulay are definied in Definition~\ref{homological defs}.

\begin{theorem}\label{S(d) thm}
Let $\bfd$ be an effective divisor on $E$ with degree $\deg\bfd=d\leq 2$. Set $R=S(\bfd)$. Then:
\begin{enumerate}[(1)]
\item $R$ is $g$-divisible with $R/gR\cong \ovl{R}=B(E,\LL(-\bfd),\sigma)$ and $R^{(3)}=S^{(3)}(\bfd+\bfd^\sigma+\bfd^{\sigma^2})$. The Hilbert series of $R$ is given by
$$h_{R}(t)=\frac{t^2+(1-d)t+1}{(t-1)^2(1-t^3)}.$$
\item $R$ is Auslander-Gorenstein and Cohen-Macaulay.
\item $R$ satisfies $\chi$ on the left and right, has cohomological dimension 2 and possesses a balanced dualizing complex.
\item $R$ is a maximal order in $Q_\gr(R)=Q_\gr(S)$.
\end{enumerate}
\end{theorem}
\begin{proof}
(1) follows from Proposition~\ref{Ro 12.2}, Corollary~\ref{S(p+q) 3 Veronese},  Corollary~\ref{Qgr(S(p+q))} and Theorem~\ref{S(p+q) g div}. The rest follows from part (1) and \cite[Proposition~2.4]{RSS2}.
\end{proof}

Unsurprisingly, it is the connection between our subalgebras, $S(\bfd)$ of $S$, and the blowup subalgebras, $T(\bfe)$ of $T$, that motivates the name ``blowup subalgebra of $S$". In particular, $S(\bfd)^{(3)}=T([\bfd]_3)$, and $S(\bfd)$ and $T(\bfe)$ satisfy very similar properties (compare Theorem~\ref{S(d) thm} with \cite[Theorem~1.1]{RSS}). Beyond this, we will see in the coming sections that the $S(\bfd)$ play a similar role to that which the $T(\bfe)$ played in \cite{RSS}.

\section{Classifying $g$-divisible maximal orders}\label{g-divisible max orders}

In this section we aim to obtain a classification of $g$-divisible maximal $S$-orders: Theorem~\ref{RSS 7.4} and Proposition~\ref{RSS 7.4(3)}. We prove that any $g$-divisible maximal $S$-order is obtain from a virtual blowup in the sense of Definition~\ref{virtual blowup}. 
Here we follow ideas of \cite{RSS}, expect one strength of our approach is that we are able to bypass many of the more technical aspects of \cite{RSS, RSS2}.

\subsection*{Endomorphism rings over the $S(\bfd)$}

We start by stating some preliminary results from \cite{RSS}.

\begin{lemma}\label{RSS 5.3}\cite[Corollary~5.3]{RSS} Let $E$ be a smooth elliptic curve, $\HH$ an invertible sheaf on $E$ with $\deg\HH\geq 2$, and $\rho:E\to E$ an automorphism of infinite order. Set $B=B(E,\HH,\rho)$. Suppose that $C$ is a cg subalgebra of $B$ with $Q_\gr(C)=Q_\gr(B)
$. Then there exist divisors $\bfx$, $\bfy$ on $E$ with $0\leq\deg\bfx<\deg\HH$, and $k\geq 0$ such that $C_n=H^0(E,\HH(-\bfy-[\bfx]_n))$ for all $n\geq k$. \qed
\end{lemma}

In addition to the statement of \cite[Proposition~5.7]{RSS} below, we will also be requiring the construction the divisor satisfying the conclusions.

\begin{prop}\label{RSS 5.7}\cite[Proposition~5.7]{RSS} Let $E$ be a smooth elliptic curve, $\HH$ an invertible sheaf on $E$ with $\deg\HH\geq 2$, and $\rho:E\to E$ an automorphism of infinite order. Set $B=B(E,\HH,\rho)$. Suppose that $C$ is a cg subalgebra of $B$ with $Q_\gr(C)=Q_\gr(B)$. Then there exists an effective divisor $\bfd$ on $E$ supported at points on distinct $\rho$-orbits, with $\deg\bfd<\deg\HH$, and such that $C$ and $B(E,\HH(-\bfd),\rho)$ are equivalent orders. Moreover, there exists a $k\geq 0$, such that 
\begin{equation}\label{RSS 5.7 eq}
C_n\subseteq H^0(E,\HH_n(-\bfd^{\rho^k}-\bfd^{\rho^{k+1}}-\dots-\bfd^{\rho^{n-1}})).
\end{equation}
for all $n\geq k$. \qed
\end{prop}
\begin{proof}[Construction of $\bfd$ from Proposition~\ref{RSS 5.7}]\renewcommand{\qedsymbol}{}
For this construction we use the notation $[\bfx]_n=\bfx+\bfx^\rho\dots+\bfx^{\rho^{n-1}}$ as well as the more compact notation $p_j=p^{\rho^j}=\rho^{-j}(p)$ for $p\in E$ and $j\in\Z$.\par
By Lemma~\ref{RSS 5.3} there exist divisors $\bfx,\bfy$ with $0\leq \deg\bfx<\deg\HH$, and $k\geq 1$ such that $C_n=H^0(E,\HH_n(-\bfy-[\bfx]_n))$ for all $n\geq k$. Fix an $\rho$-orbit $\mathbb{O}$ on $E$ and (enlarging $k$ if necessary) pick $p\in E$ such that on $\mathbb{O}$ we have
\begin{equation}\label{mathbbO}\index[n]{xo@$\bfx|_{\bbO}$}
\bfx|_{\mathbb{O}}=\sum_{i=0}^{k} a_ip_i \;\;\;\;\text{ and }\;\;\;\; \bfy|_{\mathbb{O}}=\sum_{i=0}^{k-1}b_ip_i
\end{equation}
for some $a_i,b_i\in \Z.$ For this $p$, we set
\begin{equation}\label{e_p=} e_p=\sum_{i=0}^k a_i.\index[n]{ep@$e_p$}\end{equation}
In \cite{RSS} an argument is proved showing that $e_p\geq 0$. We define the effective divisor
\begin{equation}\label{bfd=} \bfd=\sum_p e_p p,\end{equation}
where the sum is taken over one closed point $p$ for each $\rho$-orbit chosen as above. It is shown in \cite{RSS} that, for this $\bfd$ and $k$, (\ref{RSS 5.7 eq}) holds and that $C$ and $B(E,\HH(-\bfd),\rho)$ are equivalent orders.
\end{proof}

Mimicking \cite{RSS}, we use the following terms for the data appearing in the last two lemmas.

\begin{definition}\label{geo data, norm div}
Retain the hypothesis for $C \subseteq B(E,\HH,\rho)$ from Lemma~\ref{RSS 5.3}.
\begin{enumerate}[(1)]
\item Let $\bfx$, $\bfy$ and $k$ be given by Lemma~\ref{RSS 5.3}. We call the data $(E,\HH,\rho,\bfx,\bfy,k)$ \textit{geometric data for $C$}.
\item We call the divisor $\bfd$ constructed in Proposition~\ref{RSS 5.7} a \textit{$\rho$-normalised divisor for $C$}. When no confusion can occur, we simply call $\bfd$ a normalised divisor.
\end{enumerate}
\end{definition}

\begin{remark}\label{norm div not unique}
Let $C$ be as in Definition~\ref{geo data, norm div}. Suppose that $\bfd$ is a normalised divisor for $C$ constructed out of geometric data $(E,\HH,\rho,\bfx,\bfy,k)$ for $C$. We make some observations about Definition~\ref{geo data, norm div} that follow from the construction of $\bfd$.
\begin{itemize}
\item The normalised divisor $\bfd$ depends on the geometric data $(E,\HH,\rho,\bfx,\bfy,k)$ and not $C$ itself.
\item For every $\ell\geq k$, the data $(E,\HH,\rho,\bfx,\bfy,\ell)$ is also geometric data for $C$. Moreover, the normalised divisor $\bfd$ is also a normalised divisor corresponding to the data $(E,\HH,\rho,\bfx,\bfy,\ell)$.
\item For every $n\geq 0$, the divisor $\bfd^{\rho^{-n}}=\rho^n(\bfd)$ is also a normalised divisor for $C$.
\end{itemize}
\end{remark}

We now put these terms to use.

\begin{prop}\label{RSS 5.25 lemma}
Let $U$ be a $g$-divisible cg subalgebra of $S$ with $Q_\gr(U)=Q_\gr(S)$. Suppose that $\bfd$ is a $\sigma$-normalised divisor for $\ovl{U}$ corresponding to geometric data $(E,\LL,\sigma,\bfx,\bfy,k)$ of $\ovl{U}$, as constructed in Proposition~\ref{RSS 5.7}. Then
\begin{enumerate}[(1)]
\item $\ovl{U}$ and $B(E,\LL(-\bfd),\sigma)$ are equivalent orders and
\begin{equation*}\ovl{U}_n\subseteq H^0(E,\LL_n(-\bfd^{\sigma^k}-\bfd^{\sigma^{k+1}}-\dots-\bfd^{\sigma^{n-1}}))\;\text{ for all }n\geq k;\end{equation*}
\item $US(\bfd)\subseteq S_{\leq k}S(\bfd)$.
\end{enumerate}
\end{prop}
\begin{proof}
(1). Since $Q_\gr(U)=Q_\gr(S)$, $Q_\gr(\ovl{U})=Q_\gr(\ovl{S})$ by \cite[Lemma~2.10]{RSS} and Remark~\ref{RSS section 2}. Moreover, as $\deg\LL=3$, we can take $B(E,\HH,\rho)=B(E,\LL,\sigma)=\ovl{S}$ in Lemma~\ref{RSS 5.3} and Proposition~\ref{RSS 5.7}. Part (1) is therefore exactly Proposition~\ref{RSS 5.7}.

(2). Retain the notation of the construction of the $\sigma$-normalised divisor $\bfd$ from Proposition~\ref{RSS 5.7} with $(\ovl{U},\LL,\sigma)$ in place of $(C,\HH,\rho)$. In particular we use $p_j=p^{\sigma^j}$ again. Set $V=U^{(3)}$. The key here is to understand $V$ so that we can apply \cite[Proposition~5.19]{RSS}. Recall Notation~\ref{3 Veronese notation2} for $T=S^{(3)}$. The grading of $T$ we use here is $T_n=T\cap S_n$. \par
We claim that $[\bfd]_3=\bfd+\bfd^\sigma+\bfd^{\sigma^2}$ is a $\tau$-normalised divisor for $\ovl{V}$. Since $(E,\LL,\sigma,\bfx,\bfy,k)$ is geometric data for $\ovl{U}$ we have $\ovl{U}_n=H^0(E,\LL(-\bfy-[\bfx]_n))$ for all $n\geq k$. By enlarging $k$ if necessary, we may assume $k$ is divisible by 3. It then follows that $\ovl{V}_{3n}=H^0(E,\LL_{3n}(-\bfy-[\bfx]_{3n}))$ for $n\geq \frac{k}{3}$. We see $(E,\LL_3,\tau,\bfy,[\bfx]_3,k)$ is geometric data for $\ovl{V}$ ($\frac{k}{3}$ also works but there is no harm in taking the bigger $k$). Now, the $\sigma$-orbit $\mathbb{O}$ from (\ref{mathbbO}) is the disjoint union of three $\tau$-orbits, $\mathbb{O}=\mathbb{O}_0\sqcup\mathbb{O}_1 \sqcup\mathbb{O}_2$. Here $\mathbb{O}_\ell=\{ p_{\ell+3j}\,|\; j\in\Z\}$, where we are writing $\mathbb{O}=\{ p_j\,|\; j\in\Z\}$. Since $\bfx|_{\mathbb{O}}=\sum a_ip_i$,
\begin{multline*} [\bfx]_3|_{\mathbb{O}}=\sum_{i=0}^{k}a_i(p_i+p_{i+1}+p_{i+2})
\\
=a_0p_0+(a_0+a_1)p_1+\left( \sum_{i=2}^{k} (a_{i-2}+a_{i-1}+a_{i})p_{i}\right)+(a_{k}+a_{k-1})p_{k+1}+a_kp_{k+2},
\end{multline*}
 and therefore:
$$\begin{array}{l}
\,[\bfx]_3 |_{\mathbb{O}_0}=a_0p_0+(a_1+a_2+a_3)p_3+(a_4+a_5+a_6)p_6+\dots+ (a_{k-2}+a_{k-1}+a_k)p_k; \\
\,[\bfx]_3 |_{\mathbb{O}_1}=(a_0+a_1)p_1+(a_2+a_3+a_4)p_4+\dots+ (a_{k-4}+a_{k-3}+a_{k-2})p_{k-2}+(a_{k-1}+a_k)p_{k+1};\\
\,[\bfx]_3 |_{\mathbb{O}_2}=(a_0+a_1+a_2)p_2+\dots+ (a_{k-3}+a_{k-2}+a_{k-1})p_{k-1}+a_kp_{k+2}.
\end{array}$$
Put $e_p=\sum a_i$ as in (\ref{e_p=}). It follows that a $\tau$-normalised divisor for $\ovl{V}$ can be given by 
$$\bfe=\sum_p e_pp_0+e_pp_1+e_pp_2$$
where the sum ranges over the same points used to define $\bfd$ in (\ref{bfd=}). Note that indeed $p_0$, $p_1$ and $p_2$ are on different $\tau$-orbits. On the other hand, by considering different $\sigma$-orbits separately, it is not hard to see $[\bfd]_3=\bfd+\bfd^\sigma+\bfd^{\sigma^2}=\bfe$. Hence $[\bfd]_3$ is a $\tau$-normalised divisor for $\ovl{V}$ as claimed.\par

By \cite[Proposition~5.19]{RSS}, \cite[Theorem~5.3(6)]{RSS} and Theorem~\ref{S(d) thm}(1) we have
\begin{equation}\label{apply RSS 5.20} V\subseteq T_{\leq k}T([\bfd]_3)=T_{\leq k}(S(\bfd)^{(3)}).\end{equation}
Since $U\subseteq S$ is $g$-divisible, it is noetherian by \cite[Proposition~2.9]{RSS} and Remark~\ref{RSS section 2}, and therefore $U$ is a finitely generated as a right $V$-module by Lemma~\ref{noeth up n down}. By enlarging $k$ if necessary, we may assume $U_V$ is generated in degrees less than $k$. In which case, using (\ref{apply RSS 5.20}), we have
$$U=U_{\leq k}V\subseteq S_{\leq k}V\subseteq S_{\leq k} T_{\leq k}(S(\bfd)^{(3)})\subseteq S_{\leq 2k}(S(\bfd)^{(3)})$$
Hence, $US(\bfd)\subseteq S_{\leq 2k}(S(\bfd)^{(3)})S(\bfd)=S_{\leq 2k}S(\bfd)$. Since it is harmless to replace $k$ by $2k$, this proves part (2).
\end{proof}

With Proposition~\ref{RSS 5.25 lemma} we are able to prove the crucial result that links a general $g$-divisible subalgebra of $S$ to one of the $S(\bfd)$. 

\begin{theorem}\label{RSS 5.25}
Let $U\subseteq S$ be a $g$-divisible cg subalgebra of $S$ with $Q_\gr(U)=Q_\gr(S)$. Then there exists an effective divisor $\bfd$ on $E$ with $\deg\bfd\leq 2$, supported on points with distinct $\sigma$-orbits, and such that $U$ and $S(\bfd)$ are equivalent orders. \par
In more detail, for this $\bfd$, the $(U,S(\bfd))$-bimodule $M=\widehat{US(\bfd)}$ is a finitely generated $g$-divisible right $S(\bfd)$-module with $S(\bfd)\subseteq M\subseteq S$. If $W=\End_{S(\bfd)}(M)$, then $U\subseteq W\subseteq S$, $_WM$ is finitely generated and $W$, $U$, and $S(\bfd)$ are equivalent orders. The divisor $\bfd$ is any $\sigma$-normalised divisor constructed for $\ovl{U}$.
\end{theorem} 
\begin{proof}
Choose $k\geq 0$ and an effective divisor $\bfd$ on $E$, supported on distinct orbits, satisfying the conclusion of Proposition~\ref{RSS 5.25 lemma}. In particular we have $US(\bfd)\subseteq S_{\leq k}S(\bfd)$, and it follows that
$$\widehat{US(\bfd)}\subseteq \widehat{S_{\leq k}S(\bfd)}.$$
\par
Write $R=S(\bfd)$ and $M=\widehat{UR}$. Now $R$ is noetherian and $g$-divisible by Theorem~\ref{S(d) thm}(1), and $S_{\leq k}R$ is clearly a finitely generated right $R$-module. Thus $N=\widehat{S_{\leq k}R}$ is a finitely generated by \cite[Lemma~2.13(2)]{RSS} and Remark~\ref{RSS section 2}. In particular, $N$ is a noetherian $R$-module, and hence the submodule $M_R$ is finitely generated too. Put $W=\End_{R}(M)$; as $1\in M\subseteq S$, $W\subseteq S$. By \cite[Lemma~2.12(3)]{RSS} and Remark~\ref{RSS section 2}, $W$ is $g$-divisible and $_WM$ is finitely generated; it follows that $W$ and $R$ are then equivalent orders. It is easily checked that $UM\subseteq M$, and thus $U\subseteq W$. \par
Now consider the $(\ovl{W},\ovl{R})$-bimodule $\ovl{M}$. This is nonzero as $M\supseteq R$ and is finitely generated on both sides because $_WM_R$ is finitely generated on both sides; hence $\ovl{W}$ and $\ovl{R}$ are equivalent orders. By Proposition~\ref{RSS 5.25 lemma}, $\ovl{R}$ and $\ovl{U}$ are equivalent orders, and therefore $\ovl{W}$ and $\ovl{U}$ are likewise. Since $U\subseteq W\subseteq S$ we can now apply \cite[Proposition~2.16]{RSS} and Remark~\ref{RSS section 2} to conclude $U$ and $W$, and hence also $U$ and $R$, are equivalent orders.
\end{proof}

\begin{cor}\label{RSS 6.6}
Let $U\subseteq S$ be a $g$-divisible connected graded maximal $S$-order.
\begin{enumerate}[(1)]
\item There exists an effective divisor $\bfd$ on $E$ with $\deg\bfd\leq 2$, and a $g$-divisible $(U,S(\bfd))$-bimodule $M$, such that $U=\End_{S(\bfd)}(M)$. The bimodule $M$ is finitely generated on both sides and satisfies $S(\bfd)\subseteq M\subseteq S$.
\item Set $F=\End_{S(\bfd)}(M^{**})$. Then $F$ is the unique maximal order containing $U$ and $U=F\cap S$.
\end{enumerate}
\end{cor}
\begin{proof}
(1). By Theorem~\ref{RSS 5.25}, there is an effective divisor $\bfd$ with $\deg\bfd\leq 2$ and such that
$ U\subseteq W=\End_{S(\bfd)}(M)\subseteq S$, where $M=\widehat{US(\bfd)}$ is finitely generated on both sides as a $(W,S(\bfd))$-bimodule. Clearly we have $S(\bfd)\subseteq M\subseteq S$ also. By Theorem~\ref{RSS 5.25} again, $U$ and $W$ are equivalent orders. Since $U$ is a maximal $S$-order and $W\subseteq S$, $U=W$. \par
(2) By \cite[Lemma~6.3]{RSS}, $F$ is the unique maximal order equivalent to and containing $\End_{S(\bfd)}(M)=U$. Let $V=F\cap S$. Since $U\subseteq S$, we have $U\subseteq V\subseteq F$. If $x,y\in Q_\gr(S)$ are nonzero and such that $xFy\subseteq U$, then clearly $xVy\subseteq U$ also. Thus $U$ and $V$ are also equivalent orders. But $U$ is a maximal $S$-order; therefore $U=V$.
\end{proof}

The next two results investigate the situation $\End_{S(\bfd)}(M)\subseteq\End_{S(\bfd)}(M^{**})$ arising in Corollary~\ref{RSS 6.6}.

\begin{prop}\label{RSS 6.4}
Let $\bfd$ be an effective divisor with $\deg \bfd\leq 2$ and let $R=S(\bfd)$. Let $M\subseteq \Sg$ be a $g$-divisible finitely generated graded right $R$-module such that $R\subseteq M\subseteq S$. Put $W=\End_R(M)$, $F=\End_R(M^{**})$ and $V=F\cap S$. Then:
\begin{enumerate}[(1)]
\item $F$, $V$ and $W$ are $g$-divisible algebras with $Q_\gr(W)=Q_\gr(V)=Q_\gr(F)=Q_\gr(S)$.
\item $F$ is the unique maximal order containing and equivalent to $W$, while $V$ is the unique maximal $S$-order containing and equivalent to $W$.
\item There exists an ideal $K$ of $F$, contained in $W$ (and hence $V$) such that $\GK(F/K)\leq 1$.
\item $R=\End_W(M)=\End_F(M^{**})$.
\end{enumerate}
\end{prop}
\begin{proof}
This result is our analogue of \cite[Proposition~6.4]{RSS}. The proof given in \cite{RSS} is in fact sufficiently general to work in our case as well if we interchange their $T(\bfd)$ for our $S(\bfd)$. For part (3), \cite{RSS} uses that the $T(\bfd)$ satisfy the Cohen-Macaulay property; this is also true for the $S(\bfd)$ by Theorem~\ref{S(d) thm}(2).
\end{proof}



In Corollary~\ref{RSS 6.6} there is the possibility that $F\not\subseteq S$. Proposition~\ref{RSS 6.4}(3) helps us to control this situation. We give a definition for pairs of algebras satisfying the conclusions of Proposition~\ref{RSS 6.4}.

\begin{definition}\label{max order pair}\index{maximal order pair}
\item A pair $(V,F)$ of connected graded subalgebras of $\Sg$ are called \textit{a maximal order pair of $S$} if:
\begin{enumerate}[(1)]
\item $V$ and $F$ are $g$-divisible with $V\subseteq F$ and $V=F\cap S$;
\item $F$ is a maximal order in $Q_\gr(F)=Q_\gr(S)$ while $V$ is a maximal $S$-order;
\item there exists an ideal $K$ of $F$, contained in $V$, and such that $\GK(F/K)\leq 1$.
\end{enumerate}
\end{definition}

The ``of $S$" from ``maximal order pair of $S$" is there to distinguish it from the $S^{(3)}$-version \cite[Definition~6.5]{RSS}. When no ambiguity can arise it will often be dropped. We will see later in Theorem~\ref{3 Veronese of virtual blowup} that if $(V,F)$ is a maximal order pair of $S$, then $(V^{(3)},F^{(3)})$ is a maximal order pair of $S^{(3)}$.\par
Definition~\ref{max order pair} is hard to work with. We will often be using an equivalent formulation given by the following consequence of Corollary~\ref{RSS 6.6} and Proposition~\ref{RSS 6.4}.

\begin{lemma}\label{max order pair equiv def}
 Let $F$ be a cg subalgebra of $\Sg$ and $V$ a cg subalgebra of $S$. Then the following are equivalent:
\begin{enumerate}[(1)]
\item $(F,V)$ is a maximal order pair of $S$.
\item There exist an effective divisor $\bfd$ with $\deg\bfd\leq 2$ and a finitely generated $g$-divisible right $S(\bfd)$-module $M$ with $S(\bfd)\subseteq M\subseteq S$ and such that $F=\End_{S(\bfd)}(M^{**})\;\text{ and }V=F\cap S$.
\end{enumerate}
\end{lemma}

\begin{proof} (1). Suppose that $(F,V)$ is a maximal order pair of $S$. Then $V$ is $g$-divisible maximal $S$-order. Now apply Corollary~\ref{RSS 6.6} to $V$. The reverse implication follows from Proposition~\ref{RSS 6.4}(2)(3).
\end{proof}

Given a maximal order pair $(U,F)$, we next want to describe their images in $S_{(g)}/gS_{(g)}$.

\begin{prop}\label{RSS 6.7}
Let $\bfd$ be an effective divisor with $\deg\bfd\leq 2$, and let $R=S(\bfd)$. Suppose that $M$ is a finitely generated $g$-divisible right $R$-module such that $R\subseteq M\subseteq S$. Write $U=\End_R(M)$ and $F=\End_R(M^{**})$. By \cite[Theorem~1.3]{AV} there exists an effective divisor $\mbf{y}$ such that $$\ovl{M}\ehd\bigoplus_{n\geq 0}H^0(E,\OO(\mbf{y})\otimes \LL_n(-[\bfd]_n).$$
In addition to this
\begin{equation}\label{RSS 6.7 eq}
 \ovl{F}\ehd\ovl{U}\ehd B(E,\LL(-\mbf{x}),\sigma) \; \text{\, where }\mbf{x} =\bfd-\mbf{y}+\mbf{y}^\sigma
 \end{equation}
 holds. Moreover, if $V=F\cap S$, then $U\subseteq V\subseteq F$, and $(V,F)$ is a maximal order pair.
\end{prop}
\begin{proof}
This is proved in the same way as \cite[Theorem~6.7]{RSS}. For this to work, one must bear in mind $S(\bfd)$ is Auslander-Gorenstien and Cohen-Macaulay by Theorem~\ref{S(d) thm}(2), and that the relevant results \cite[Section~2]{RSS} hold with $S$ and $T$ interchanged. The reader can find a detailed write out of this in \cite[5.14-5.17]{thesis}.
\end{proof}

\subsection*{Blowups and virtual blowups}

The algebra $F$ from Proposition~\ref{RSS 6.7} has many similar properties with the $S(\bfd)$. Most notable is that $\ovl{F}\cong F/gF$ is equal in high degrees to a twisted homogeneous coordinate ring. We will be calling such $F$ \textit{virtual blowups} (Definition~\ref{virtual blowup}).\par
We now investigate precisely what divisors $\bfx$ arise in Proposition~\ref{RSS 6.7}. These will be the divisors at which we can ``blow up".

\begin{definition}\label{virtually effective}\cite[Definition~7.1]{RSS}.\index{virtually effective divisor}
Let $\rho:E\to E$ be an automorphism of infinite order. A divisor $\bfx$ on $E$ with $\deg\bfx\geq 0$ is called \textit{$\rho$-virtually effective} if there exists $n\geq 0$ such that $\bfx+\bfx^\rho+\dots+\bfx^{\rho^{n-1}}$ is an effective divisor on $E$.
\end{definition}

When the automorphism $\rho:E\to E$ is clear from context, we often drop the $\rho$ from $\rho$-virtually effective.

\begin{lemma}\label{RSS 7.3(1)}
The divisor $\mbf{x}$ in (\ref{RSS 6.7 eq}) is $\sigma$-virtually effective.
\end{lemma}
\begin{proof}
Let $\mbf{x}$, $U$, $F$ be as in Proposition~\ref{RSS 6.7}. Put $B(E,\NN,\sigma)$ where $\NN=\LL(-\bfx)$. By Proposition~\ref{RSS 6.7}, $\ovl{U}\ehd \ovl{F}\ehd B$. Let $n\geq 0$ be such that  $\ovl{U}_n=B_n=H^0(E,\NN_n)$. Then because $\ovl{U}\subseteq \ovl{S}=B(E,\LL,\sigma)$, we have that $H^0(E,\NN_n)\subseteq H^0(E,\LL_n)$. Enlarging $n$ if necessary, we can assume both $\deg\LL_n\geq 2$ and $\deg\NN_n>2$. Then \cite[Corollary~IV.3.2]{Ha} implies $\LL_n$ and $\NN_n$ are generated by their global sections. Therefore we have $\NN_n=\LL_n(-[\mbf{x}]_n)\subseteq \LL_n$, and hence $[\mbf{x}]_n$ is effective.
\end{proof}

\begin{remark}
Our definition of a virtually effective divisor differs from \cite[Definition~7.1]{RSS}. It is implicit in \cite[Proposition~7.3(2)]{RSS} that they are equivalent definitions. We prefer Definition~\ref{virtually effective} for it is simpler to state and better motivates the name ``virtually effective".
\end{remark}

To avoid confusion, we now reserve the notation $[\bfx]_n=\bfx+\bfx^\sigma+\dots+\bfx^{\sigma^{n-1}}$ from Notation~\ref{[d]_n} for the automorphism $\sigma$ from Hypothesis~\ref{standing assumption intro}. This is particularly relevant in the next lemma.

\begin{lemma}\label{[x]_3 is v effective}
Let $\bfx$ be a $\sigma$-virtually effective divisor. Then $[\bfx]_3=\bfx+\bfx^\sigma+\bfx^{\sigma^2}$ is $\tau$-virtually effective, where $\tau=\sigma^3$.
\end{lemma}
\begin{proof}
Suppose that $[\bfx]_n$ is effective. If necessary, enlarge $n$ so that $n$ is divisible by 3, say $n=3m$. We then have
\begin{multline*}
[\bfx]_3+[\bfx]_3^\tau+\dots+[\bfx]_3^{\tau^{m-1}}=(\mbf{x}+\mbf{x}^\sigma+ \mbf{x}^{\sigma^2})+(\mbf{x}+\mbf{x}^\sigma+\mbf{x}^{\sigma^2})^{\tau}+
\dots+(\mbf{x}+\mbf{x}^\sigma+\mbf{x}^{\sigma^2})^{\tau^{m-1}}\\
=\bfx+\bfx^\sigma+\dots+\bfx^{\sigma^{n-1}}=[\mbf{x}]_n.
\end{multline*}
This is effective by assumption; hence $[\bfx]_3$ is $\tau$-virtually effective.
\end{proof}

We now turn our attention away from divisors and back to subalgebras of $S$. We start with the definition of a virtual blowup. 

\begin{definition}\label{virtual blowup}\index{virtual blowup of $S$}
Let $\bfx$ be a $\sigma$-virtually effective divisor on $E$ with $\deg\bfx\leq 2$. We say that a cg subalgebra $F$ of $\Sg$ with $Q_\gr(F)=Q_\gr(S)$ is a \textit{virtual blowup of $S$ at $\bfx$} if:
\begin{enumerate}[(1)]
\item $F$ is a part of a maximal order pair $(F\cap S,F)$ of $S$.
\item $\ovl{F}\ehd B(E,\LL(-\bfx),\sigma)$.
\end{enumerate}
\end{definition}

\begin{remark}\label{why vblowup}
Like with the $S(\bfd)$, our main motivation for the name ``virtual blowup" is by analogy with \cite{RSS}. On top of this we will also prove in Theorem~\ref{3 Veronese of virtual blowup} that if $F$ is a virtual blowup of $S$, then $F^{(3)}$ is a virtual blowup of $T=S^{(3)}$. Justification for their case is given in \cite[Remark~7.5]{RSS}.
\end{remark}

With our new language we can give our main result of this section.

\begin{theorem}\label{RSS 7.4}
\begin{enumerate}[(1)]
\item Let $V\subseteq S$ be a $g$-divisible cg maximal $S$-order. Then:
\begin{enumerate}[(a)]
\item there is a maximal order $F\supseteq V$ such that $(V,F)$ is a maximal order pair;
\item $F$ is a virtual blowup of $S$ at a virtually effective divisor $\mbf{x}$ with $\deg\mbf{x}\leq 2$;
\item $\ovl{V}\ehd\ovl{F}\ehd B(E,\LL(-\mbf{x}),\sigma)$.
\end{enumerate}
\item If $U\subseteq S$ is any $g$-divisible cg subalgebra with $Q_\gr(U)=Q_\gr(S)$, then there exists a maximal order pair $(V,F)$ as in (1), such that $U$ is contained in, and equivalent, to $V$ and $F$.
\end{enumerate}
\end{theorem}

\begin{proof}
(1). By definition $Q_\gr(V)=Q_\gr(S)$. By Corollary~\ref{RSS 6.6}(2) and Lemma~\ref{max order pair equiv def}, $V$ is part of a maximal order pair $(V,F)$. This proves part (a). By Lemma~\ref{max order pair equiv def}, $F=\End_{S(\bfd)}(M^{**})$ for some effective divisor $\bfd$ with $\deg\bfd\leq 2$, and a finitely generated $g$-divisible right $S(\bfd)$-module $M$ satisfying $S(\bfd)\subseteq M\subseteq S$. By Proposition~\ref{RSS 6.7} and Lemma~\ref{RSS 7.3(1)}, $F$ is a virtual blowup of $S$ at a virtually effective divisor $\bfx$. Part ($c$) follows directly from part ($b$).\par

(2). By Theorem~\ref{RSS 5.25}, $U$ is contained in and equivalent to some $\End_{S(\bfd)}(M)$, where $\bfd$ is an effective divisor of $\deg\bfd\leq 2$ and $M=\widehat{US(\bfd)}$. Clearly $S(\bfd)\subseteq M\subseteq S$ and so we can apply Proposition~\ref{RSS 6.4}.
\end{proof}

Just as in \cite{RSS}, in addition to a complete description of $g$-divisible maximal $S$-orders we are able to obtain a description of any $g$-divisible subalgebra. 

\begin{cor}\label{RSS 7.6}
Let $U\subseteq S$ be a $g$-divisible subalgebra with $Q_\gr(U)=Q_\gr(S)$. Then $U$ is an iterated sub-idealiser inside a virtual blowup of $S$. More precisely the following holds.
\begin{enumerate}[(1)]
\item There exists a virtually effective divisor $\mbf{x}=\mbf{u}-\mbf{v}+\mbf{v^\sigma}$ with $\deg\mbf{x}\leq 2$, and a blowup $F$ of $S$ at $\mbf{x}$, such that $V=F\cap S$ contains, and is equivalent to, $U$. The pair $(V,F)$ is a maximal order pair.
\item There is a $g$-divisible algebra $W$ with $U\subseteq W\subseteq V$ and such that $U$ is a right sub-idealiser inside $W$, and $W$ is a left sub-idealiser inside $V$. In more detail:
    \begin{enumerate}[(a)]
    \item There exists a $g$-divisible left ideal $L$ of $V$ such that either $L=V$ or $V/L$ is $2$-pure. There exists a $g$-divisible ideal $K$ of $X=\I_V(L)$ such that $K\subseteq W\subseteq X$ and $\GK_X(X/K)\leq 1$.
    \item $V$ is a finitely generated left $W$-module, while $X/K$ is a finitely generated $\Bbbk[g]$-module. In particular $X$ is finitely generated over $W$ on both sides;
    \item The properties given for $W\subseteq V$ also hold true for the pair $U\subseteq W$ with left and right interchanged. 
    \end{enumerate}
\end{enumerate}
\end{cor}
\begin{proof}
The proof of \cite[Corollary~7.6]{RSS} works here.
\end{proof}

Theorem~\ref{RSS 7.4} shows any $g$-divisible maximal $S$-order is a virtual blowup at some virtually effective divisor. Conversely, we now show that given any virtually effective divisor $\bfx$ with $\deg\bfx\leq 2$, there exists a virtual blowup $F$ of $S$ at $\bfx$, where $V=F\cap S$ is necessarily a maximal $S$-order. This completes the classification of $g$-divisible maximal $S$-orders. For the proof we need to recall the rings $T(\bfd)$ from Definition~\ref{T(d) def}.

\begin{lemma}\label{RSS2 5.10}
Let $\mbf{u}$ and $\mbf{v}$ be effective divisors on $E$ such that $\deg\mbf{u}\leq 2$ and $\mbf{v}\leq [\mbf{u}]_k$ for some $k\geq 1$. Then there exists a $g$-divisible right $S(\mbf{u})$-module $M$ with $S(\mbf{u})\subseteq M\subseteq S$ and such that
\begin{equation}\label{ehd 1 RSS2 5.10} \ovl{M}\ehd\bigoplus_{n\geq 0}H^0(E,\LL_n(-[\mbf{u}]_n+\mbf{v})).\end{equation}
\end{lemma}
\begin{proof}
In this proof $T$ is graded via $T_n=T\cap S_n$. By Theorem~\ref{S(d) thm}(1) with $\mbf{u}=\bfd$ we have $S(\mbf{u})^{(3)}=T([\mbf{u}]_3)$. By  \cite[Lemma~5.10]{RSS2} there exists a $g$-divisible right $T([\mbf{u}]_3)$-module $N$ with $T([\mbf{u}]_3)\subseteq N\subseteq T$ and such that
\begin{equation}\label{RSS 7.4(3) eq1} \ovl{N}\ehd
\bigoplus_{m\geq0}H^0(E,\LL_{3m}(-[\mbf{u}]_{3m}+\mbf{v})).\end{equation}
Set $R=S(\mbf{u})$ and $M=\widehat{NR}$. Then $M$ is a $g$-divisible right $R$-module. As $N$ is a finitely generated right $R^{(3)}$-module, $NR$ is a finitely generated right $R$-module. By \cite[Lemma~2.13(2)]{RSS} and Remark~\ref{RSS section 2}, $M_R$ is then finitely generated also. Further, as $1\in N\subseteq NT=T$, we have $R\subseteq M\subseteq S$. This leaves (\ref{ehd 1 RSS2 5.10}) to be proven.\par
Now because $N$ is $g$-divisible and clearly $(NR)^{(3)}=N$, it is easy to see $M^{(3)}=N$. On the other hand, since $\ovl{R}=B(E,\LL(-\mbf{u}),\sigma)$, for all $n\gg 0$, $\ovl{M}_n=H^0(E,\LL_{n}(-[\mbf{u}]_{n}+\mbf{w}))_n$ for some divisor $\mbf{w}$ by \cite[Theorem~1.3]{AV}. Equation (\ref{RSS 7.4(3) eq1}) and $M^{(3)}=N$ then implies
\begin{equation}\label{gb secs RSS 5.10} H^0(E,\LL_{3m}(-[\mbf{u}]_{3m}+\mbf{v}))=H^0(E,\LL_{3m}(-[\mbf{u}]_{3m}+\mbf{w}))\;
\text{ for }m\gg0.\end{equation}
Enlarging $m$ if necessary, we can assume $\LL_{3m}(-[\mbf{u}]_{3m}+\mbf{v}))$ and $\LL_{3m}(-[\mbf{u}]_{3m}+\mbf{w}))$ are generated by their global sections. Thus (\ref{gb secs RSS 5.10}) implies that $\LL_{3m}(-[\mbf{u}]_{3m}+\mbf{v}))=\LL_{3m}(-[\mbf{u}]_{3m}+\mbf{w}))$. Hence we get $\mbf{v}=\mbf{w}$, proving (\ref{ehd 1 RSS2 5.10}).
\end{proof}

\begin{prop}\label{RSS 7.4(3)}
Let $\mbf{x}$ be a $\sigma$-virtually effective divisor with $\deg\mbf{x}\leq 2$. Then there exists a blowup $F$ of $S$ at $\mbf{x}$.
\end{prop}
\begin{proof} Using \cite[Proposition~7.3]{RSS}, write $\mbf{x}=\mbf{u}-\mbf{v}+\mbf{v}^\sigma$ where $\mbf{u}$ and $\mbf{v}$ are effective and such that $0\leq\mbf{v}\leq [\mbf{u}]_k$ for some $k\geq 0$. By Lemma~\ref{RSS2 5.10} there exists a $g$-divisible right $S(\bfd)$-module $M$ with $S(\bfd)\subseteq M\subseteq S$ and such that $\ovl{M}_n=H^0(E,\LL_n(-[\mbf{u}]_n+\mbf{v}))$ for all $n\gg 0$. Put $U=\End_R(M)$ and $F=\End_R(M^{**})$. By Proposition~\ref{RSS 6.7}, $\ovl{F}\ehd\ovl{U}\ehd B(E,\LL(-\bfx),\sigma)$. By Lemma~\ref{max order pair equiv def}(1), $(F\cap S,F)$ is a maximal order pair.
\end{proof}

\subsection*{3-Veroneses of virtual blowups}\label{further vblowups}

We end this section by proving that the 3-Veroneses of our virtual blowups are virtual blowups of $S^{(3)}$.

\begin{lemma}\label{in B is Veroneses equal then ehd}
Let $B=B(E,\NN,\sigma)$ for some invertible sheaf $\NN$ on $E$ with $\deg\NN\geq 1$. Let $M,N$ be two finitely generated right $B$-submodules of $Q_\gr(B)=\Bbbk(E)[t,t^{-1};\sigma]$. If $M^{(d)}=N^{(d)}$ for some $d\geq 1$, then $M\ehd N$.
\end{lemma}
\begin{proof}
By \cite[Theorem~1.3]{AV}, there exists divisors $\bfx$ and $\bfy$ such that $M_n=H^0(E,\OO(\bfx)\otimes\NN_n)$ and $N_nH^0(E,\OO(\bfy)\otimes\NN_n)$ for $n\gg0$. Let $n\geq0$. By assumption $M_{dn}=N_{dn}$, thus $H^0(E,\OO(\bfx)\otimes\NN_{dn})=H^0(E,\OO(\bfy)\otimes\NN_{dn})$. By enlarging $n$ if necessary, we can assume that both $\deg(\OO(\bfx)\otimes\NN_{dn})\geq 2$ and $\deg(\OO(\bfy)\otimes\NN_{dn})\geq 2$. In which case both sheaves are generated by their global sections by \cite[Corollary IV.3.2]{Ha}. Hence $\OO(\bfx)\otimes\NN_{dn}=\OO(\bfy)\otimes\NN_{dn}$. It follows that $\bfx=\bfy$, and then $M\ehd N$.
\end{proof}

\begin{lemma}\label{NR**=M**}
Let $R=S(\bfd)$ for some effective divisor of $\deg\bfd\leq 2$, and $M\subseteq\Sg$ be a finitely generated $g$-divisible right $R$-module. Let $d\geq 1$ and write $N=M^{(d)}$. Then $(NR)^{**}=M^{**}.$
\end{lemma}
\begin{proof} Now since $(NR)^{(d)}=N=M^{(d)}$, $(\ovl{NR})^{(d)}=\ovl{M}^{(d)}$. So by Lemma~\ref{in B is Veroneses equal then ehd} $\ovl{NR}\ehd\ovl{M}$. 
Set $X=M/(NR)$. Since $M$ is $g$-divisible, $X/gX \cong \ovl{M}/\ovl{NR}$, and thus $\dim_\Bbbk(X/gX)<\infty$. This forces $\dim_\Bbbk X_n \geq \dim_\Bbbk X_{n+3}$ for all $n \gg0$, and therefore $\GK(M/(NR))\leq 1$. By Theorem~\ref{S(d) thm}(2), $R$ is Cohen-Macaulay. Therefore, as $\GK(M/NR)\leq 1 =\GK(R)-2$, we have that $M\subseteq (NR)^{**}$ by \cite[Lemma~4.11(1)]{RSS}. It then follows $M^{**}\subseteq (NR)^{**}$. Since $NR\subseteq M$, we also have the reverse inclusion $(NR)^{**}\subseteq M^{**}$.
\end{proof}

\begin{lemma}\label{Veronese ** commute}
Let $\bfd$ be an effective divisor with $\deg\bfd\leq 2$ and let $R=S(\bfd)$. Suppose that $R\subseteq M\subseteq S$ is a finitely generated $g$-divisible right $R$-module. Then $(M^{**})^{(3)}=(M^{(3)})^{**}$.
\end{lemma}
\begin{proof}
Let $R'=R^{(3)}$ and $N=M^{(3)}$. By Theorem~\ref{S(d) thm}(2), $R$ is Auslander-Gorenstein and Cohen-Macaulay. Hence by \cite[Lemma~4.11(1)]{RSS}, $\GK_{R}(M^{**}/M)\leq 1$. Lemma~\ref{noeth up n down}(2) implies that $M^{**}/M$ is finitely generated as a right $R'$-module, hence $\GK_{R'}(M^{**}/M)=\GK_{R}(M^{**}/M)$. As right $R'$-modules clearly $(M^{**})^{(3)}/N=(M^{**}/M)^{(3)}\subseteq M^{**}/M$, and so we have that
\begin{equation}\label{GK(M**(3)/N)} \GK_{R'}((M^{**})^{(3)}/N)\leq \GK_{R'}(M^{**}/M)\leq 1.\end{equation}
Now $R'=T([\bfd]_3)$ by Theorem~\ref{S(d) thm}(1), so $R'$ is both Auslander-Gorenstein and Cohen-Macaulay by \cite[Theorem~1.1]{Ro}. Hence by \cite[Lemma~4.11(1)]{RSS} applied to $R'$, $N$, and by (\ref{GK(M**(3)/N)}), $(M^{**})^{(3)}\subseteq N^{**}$.\par

Conversely, consider $L=N^{**}R/NR$ as a right $R'$-module. Clearly $L_{R'}$ is finitely generated and is a homomorphic image of $(N^{**}/N)\otimes_{R'} R$. We therefore have
$$\GK_R(L)=\GK_{R'}(L)\leq \GK_{R'}(N^{**}/N\otimes_{R'} R)\leq \GK_{R'}(N^{**}/N)\leq 1,$$
where the second inequality follows from \cite[Proposition~5.6]{KL}, and the third from \cite[Lemma~4.11(1)]{RSS}. By \cite[Lemma~4.11(1)]{RSS} applied to $L$ and $R$, and by Lemma~\ref{NR**=M**}, $N^{**}R\subseteq (NR)^{**}=M^{**}$. Taking the 3rd Veronese gives the reverse inclusion $N^{**}\subseteq (M^{**})^{(3)}$.
\end{proof}

For the next result we recall Notation~\ref{3 Veronese notation2} for $T=S^{(3)}$.

\begin{prop}\label{3 Veronese of virtual blowup}
Let $F$ be a virtual blowup of $S$ at a $\sigma$-virtually effective divisor $\bfx$ and set $F'=F^{(3)}$. Then $F'$ is a virtual blowup of $T$ at the $\tau$-virtually effective divisor $\bfy=\bfx+\bfx^\sigma+\bfx^{\sigma^2}$.
\end{prop}
\begin{proof}
First we note that $F$ indeed exists by Proposition~\ref{RSS 7.4(3)}, while $\bfy$ is $\tau$-virtually effective by Corollary~\ref{[x]_3 is v effective}. \par
On one hand $\ovl{F}\ehd B(E,\LL(-\bfx),\sigma)$ by definition; while on the other,
\begin{multline*} B(E,\LL(-\bfx),\sigma)_{3n}=H^0(E,\LL(-\bfx)_{3n})= H^0(E,\LL_{3n}(-\bfx-\bfx^\sigma-\dots-\bfx^{\sigma^{3n-1}}))\\
=H^0(E,\MM_{n}(-\bfy-\bfy^\tau-\dots-\bfy^{\tau^{n-1}}))= H^0(E,\MM(-\bfy)_n)=B(E,\MM(-\bfy),\tau)_n.
\end{multline*}
for all $n\geq 1$. Thus $\ovl{F}^{(3)}\ehd B(E,\MM(-\bfy),\tau)$. It is hence left to prove that $(F'\cap T,F')$ is a maximal order pair of $T$. \par
Apply Lemma~\ref{max order pair equiv def}(1) to the maximal order pair $(F\cap S,F)$ of $S$. We get an effective divisor $\bfd$ with $\deg\bfd\leq 2$, and a finitely generated $g$-divisible right $S(\bfd)$-module $M$ such that $S(\bfd)\subseteq M\subseteq S$ and $F=\End_{S(\bfd)}(M^{**})$. Let $R=S(\bfd)$, $R'=R^{(3)}$ and $N=M^{(3)}$. Since $M_R$ is $g$-divisible and finitely generated, $N_{R'}$ is also by Lemma~\ref{g div up} and Lemma~\ref{noeth up n down}(2). In addition, $R'\subseteq N\subseteq T$ easily follows from $R\subseteq M\subseteq S$. By Theorem~\ref{S(d) thm}(1) we know that $R'=T(\bfe)$ where $\bfe=[\bfd]_3$ has $\deg\bfe=3\deg\bfd\leq 6$. We claim $F'$, $R'$, $N$ satisfy $F'=\End_{R'}(N^{**})$.\par
Let $G=\End_{R'}(N^{**})$. Note that since $N^{**}=(M^{**})^{(3)}$ by Lemma~\ref{Veronese ** commute},
\begin{equation}\label{F' subset G}F'N^{**}\subseteq FM^{**}\cap T_{(g)}=M^{**}\cap T_{(g)}=N^{**}.
\end{equation}
Hence $F'\subseteq G$. Now consider $W=\End_R(N^{**}R)$ and $W'=W^{(3)}$. Similar to (\ref{F' subset G}) (with $N^{**}R$ in place of $M^{**}$) one can show $W'\subseteq G$. Clearly also $GN^{**}R\subseteq N^{**}R$. This shows $G\subseteq W$, and hence $G=W'$. Now by \cite[Theorem~2.7]{Coz}, $\End_R((N^{**}R)^{**})$ is the unique maximal order equivalent to and containing $W$. But by Lemma~\ref{Veronese ** commute} and Lemma~\ref{NR**=M**} (with $d=3$ and $M$ replaced by $M^{**}$),
$$(N^{**}R)^{**}=((M^{**})^{(3)}R)^{**}=(M^{**})^{**}=M^{**}.$$
So we have $W\subseteq F$, and thus $W'=G\subseteq F'$. So $F'=\End_{R'}(N^{**})$. By a $S^{(3)}$-version of Lemma~\ref{max order pair equiv def} (essentially \cite[Corollary~6.6(1)]{RSS} and \cite[Proposition~6.4]{RSS}) $(F'\cap T,F')$ is a maximal order pair of $T$.
\end{proof}

\section{The main classification}

Here we present our main classification of maximal $S$-orders (Theorem~\ref{RSS 8.11}). It turns out that we already know of all the maximal $S$-orders $U$ (such that $\ovl{U}\neq \Bbbk$). We will show that they are all $g$-divisible, and hence fit into our classification of $g$-divisible maximal $S$-orders (Theorem~\ref{RSS 7.4}).

\begin{prop}\label{C to C hat}
Let $C$ be a cg subalgebra of $S$ satisfying $Q_\gr(C)=Q_\gr(S)$ and $\ovl{C}\neq\Bbbk$. Suppose that $g\in C$. Then for all $m\gg 0$, $C\cap Sg^m=g^m\widehat{C}$. Moreover, if $C$ is noetherian, then $\widehat{C}$ is also finitely generated on both sides as a $C$-module. \qed
\end{prop}
\begin{proof}
The proof for \cite[Proposition~8.7]{RSS} is general enough to prove this statement also.
\end{proof}

Clearly Proposition~\ref{C to C hat} proves that $C$ and $\widehat{C}$ are equivalent orders. We now aim to prove $U$ and $U\langle g\rangle$ are equivalent orders. Key to this are sporadic and minimal sporadic ideals.

\begin{definition}\label{min sporadic ideal}\index{sporadic ideal}\index{minimal sporadic ideal}
Let $R$ be a cg $\Bbbk$-subalgebra of $\Sg$.
\begin{enumerate}[(1)]
\item  A homogeneous ideal $I$ of $R$ will be called \textit{sporadic} if $\GK(R/I)\leq 1$.
\item A sporadic ideal $K$ of $R$ is called \textit{minimal sporadic} if for all sporadic ideals $I$ of $R$, there exists an $n\geq 0$ such that $K_{\geq n}\subseteq I$.
\end{enumerate}
\end{definition}

We warn that Definition~\ref{min sporadic ideal}(1) does not completely agree with the definition given in \cite{RSS}. There, an ideal $I$ of $R$ was called sporadic when $\GK(R/I)=1$. It is Definition~\ref{min sporadic ideal}(1) that is most convenient for us.

\begin{remark}\label{RSS assumption 8.2}
In \cite[Section~8]{RSS} the authors work under increased assumptions designed to ensure the existence of minimal sporadic ideals \cite[Assumption~8.2]{RSS}. It is not immediately obvious that $S^{(3)}$ satisfies these assumptions: this follows from \cite[Theorem~8.8 and Proposition~8.7]{RSS2}.
\end{remark}

The coming few results show that the algebras we consider do indeed have minimal sporadic ideals.

\begin{prop}\label{RSS 8.4}
Let $R$ be a cg $g$-divisible subalgebra of $S$ with $Q_\gr(R)=Q_\gr(S)$. Then $R$ has a minimal sporadic ideal.
\end{prop}
\begin{proof}
Set $R'=R^{(3)}$. By Lemma~\ref{g div up}, $R'$ is $g$-divisible subalgebra of $T=S^{(3)}$. By Remark~\ref{RSS assumption 8.2} and \cite[Proposition~8.4]{RSS}, $R'$ has a minimal sporadic ideal, $L$ say. Set $K=RLR$. Now $R$ is noetherian by \cite[Proposition~2.9]{RSS} and Remark~\ref{RSS section 2}, hence Lemma~\ref{sporadics up n down}(2) applies to show that $\GK(R/K)\leq 1$. We claim that $K$ is in fact a minimal sporadic ideal of $R$.\par

Suppose that $I$ is any other sporadic ideal of $R$. Then by Lemma~\ref{sporadics up n down}(1), $J=I^{(3)}$ is a sporadic ideal of $R'$. Since $L$ is a minimal sporadic ideal of $R$, there exists $m\geq 0$ such that $L_{\geq m}\subseteq J$. We then have $RL_{\geq m}R\subseteq RJR\subseteq I$. But  $RL_{\geq m}R\ehd RLR = K$, hence there exists $n\geq 0$ such that $K_{\geq n}\subseteq RL_{\geq m}R\subseteq I$ as required.
\end{proof}

Proposition~\ref{RSS 8.4} clearly implies that the $S(\bfd)$ have minimal sporadic ideals. It can also be proved that virtual blowups of $S$ have minmal sporadic ideal \cite[Corollary~6.6]{thesis}, although we do not require it here. Corollary~\ref{RSS 8.8} below is the closest we can currently get to proving every algebra that contains $g$ has a minimal sporadic ideal.

\begin{cor}\label{RSS 8.8} Let $C$ be a cg subalgebra of $S$ with $Q_\gr(C)=Q_\gr(S)$. Assume that $g\in C$ and that $\ovl{C}\neq\Bbbk$. By Proposition~\ref{RSS 8.4}, $\widehat{C}$ has a minimal sporadic ideal $J$. Then $K=C\cap \widehat{J}$ is sporadic ideal minimal among sporadic ideals $I$ such that $C/I$ is $g$-torsionfree (i.e if $c\in C$ is such that $cg^n\in I$ then $c=0$).
\end{cor}
\begin{proof}
This follows in the same way as \cite[Corollary~8.8]{RSS}, with Proposition~\ref{RSS 8.4} in place of \cite[Proposition~8.4]{RSS}.
\end{proof}


Finally we give the key result that shows that $U$ and $U\langle g\rangle$ are equivalent orders.

\begin{prop}\label{RSS 8.10}
Let $U$ be a cg subalgebra of $S$ with $\ovl{U}\neq \Bbbk$ and $D_\gr(U)=D_\gr(S)$. Then there exists a nonzero ideal of $C=U\langle g\rangle$ that is finitely generated both as a right and left $U$-module.
\end{prop}
\begin{proof}
By Lemma~\ref{Qgr(S) or Qgr(T)} either $Q_\gr(C)=Q_\gr(S)$ or $Q_\gr(C)=Q_\gr(S^{(3)})$. The case $Q_\gr(C)=Q_\gr(S^{(3)})$, then this is exactly \cite[Proposition~8.10]{RSS}. So we assume that $Q_\gr(C)=Q_\gr(S)$. Despite this, the proof is almost identical to that of \cite[Proposition~8.10]{RSS}. Of course one must replace their preliminary results that they referenced with the appropriate results from this paper. The details can be found in \cite[Proposition~6.9]{thesis}.
\end{proof}

Proposition~\ref{RSS 8.10} shows that $U$ and $U\langle g\rangle$ are equivalent orders; the details can be found in the proof of the next theorem. We are now ready to state and prove our main result. In the coming results, $T=S^{(3)}$ is regraded so that $T^{(d)}=S^{(3d)}$. For completeness, we incorporate\cite[Theorem~8.11]{RSS} below. 

\begin{theorem}\label{RSS 8.11}
Let $d\geq 1$, and suppose that $U$ is a cg maximal $S^{(d)}$-order satisfying $\ovl{U}\neq \Bbbk$.
\begin{enumerate}[(1)]
\item If $d$ is coprime to 3, then there exists a $\sigma$-virtually effective divisor $\bfx$ satisfying $\deg\bfx\leq 2$, and a virtual blowup $F$ of $S$ at $\bfx$, such that $U=(F\cap S)^{(d)}$.
\item \cite[Theorem~8.11]{RSS} If $d$ is divisible by $3$, say $d=3e$, then exists a $\sigma^3$-virtually effective divisor $\bfx$ with $\deg\bfx\leq 8$, and a virtual blowup $F$ of $T$ at $\bfx$, such that $U=(F\cap T)^{(e)}$.
\end{enumerate}
\end{theorem}
\begin{proof}
(1). By definition, $Q_\gr(U)=Q_\gr(S^{(d)})=Q_\gr(S)^{(d)}$, and so $D_\gr(U)=D_\gr(S)$. Hence Proposition~\ref{RSS 8.10} applies, and we get an ideal $I$ of $C=U\langle g\rangle$ that is finitely generated on both sides as a $U$-module. Now, as both a right and left $U$-module,
$$I\cong I^{(d)}\oplus I^{(1\;\mathrm{mod}\, d)}\oplus \dots\oplus I^{(d-1\;\mathrm{mod}\, d)},$$
where $I^{(i\;\mathrm{mod}\, d)}=\bigoplus_{j\in\N}I_{i+dj}$. Because $J=I^{(d)}$ is a direct summand of $I$, it is also finitely generated on both sides as a $U$-module. Clearly $J$ is also an ideal of $C'=C^{(d)}$.\par
Set $\widetilde{U}=U+J$, then $\widetilde{U}$ is finitely generated on both sides as a $U$-module. It follows that $U$ and $\widetilde{U}$  are equivalent orders. Clearly $\widetilde{U}$ and $C'$ are also (via the common ideal $J$), so $U$ and $C'$ are equivalent orders. \par
Now consider $C$. By Lemma~\ref{Qgr(S) or Qgr(T)}, either $Q_\gr(C)=Q_\gr(S)$ or $Q_\gr(C)=Q_\gr(T)$. In fact $Q_\gr(C)=Q_\gr(T)$ is impossible because $Q_\gr(U)=Q_\gr(S)^{(d)}$ and $d$ is coprime to 3, thus $Q_\gr(C)=Q_\gr(S)$. By Proposition~\ref{C to C hat}, $C$ and $\widehat{C}$ have a common ideal. By Theorem~\ref{RSS 7.4}(2), $\widehat{C}$ is contained in, and an equivalent order to, some virtual blowup $F$ at a virtually effective divisor $\bfx$ with $\deg\bfx\leq 2$. Hence $V=F\cap S$ is a maximal $S$-order containing and equivalent to $C$. By Lemma~\ref{equiv orders go up}, $C'$ (and hence $U$) is contained in and equivalent to $V^{(d)}$ and $F^{(d)}$. As $U$ is a maximal $S^{(d)}$-order we have $U=C'=V^{(d)}$.\par
(2). This is \cite[Theorem~8.11]{RSS} with \cite[Theorem~8.8 and Proposition~8.7]{RSS2} proving that $T$ indeed satisfies the hypotheses of those results.
\end{proof}

In the other direction we obtain Theorem~\ref{main thm converse}. We note that Theorem~\ref{main thm converse}(2)(a-c) is an improvement on \cite{RSS}.

\begin{theorem}\label{main thm converse}
Let $d\geq 1$ and $\bfx$ a divisor on $E$.
\begin{enumerate}[(1)]
\item If $d$ is coprime to 3 and $\bfx$ is $\sigma$-virtually effective with $\deg\bfx\leq 2$, then there exists a blowup $F$ of $S$ at $\mbf{x}$. Moreover:
 \begin{enumerate}[(a)]
 \item $F'=F^{(d)}$ is a maximal order with $Q_\gr(F)=Q_\gr(S^{(d)})$;
 \item $V'=F'\cap S$ is a maximal $S^{(d)}$-order;
 \item there exists an ideal $L$ of $F'$, contained in $V'$, such that $\GK(F'/L)\leq 1$.
 \end{enumerate}
\item If $\bfx$ is $\tau$-virtually effective with $\deg\bfx\leq 8$, then there exists a blowup $F$ of $T$ at $\mbf{x}$. Moreover:
\begin{enumerate}[(a)]
\item $F'=F^{(d)}$ is a maximal order with $Q_\gr(F')=Q_\gr(T^{(d)})$;
\item $V'=F'\cap T$ is a maximal $T^{(d)}$-order;
\item there exists an ideal $L$ of $F'$, contained in $V'$, such that $\GK(F'/L)\leq 1$.
\end{enumerate}
\end{enumerate}
\end{theorem}
\begin{proof}
(1). By Proposition~\ref{RSS 7.4(3)}, a blowup $F$ of $S$ at $\mbf{x}$ exists. In particular $(F,V)$, where $V=F\cap S$, is a maximal order pair of $S$. Since $F$ must be $g$-divisible, Proposition~\ref{g-div max orders up n down}(1a) implies $F'$ is a maximal order in $Q_\gr(S^{(d)})$. Similarly $V'$ is a maximal $S^{(d)}$-order by Proposition~\ref{g-div max orders up n down}(1b). By definition of a maximal order pair (Definition~\ref{max order pair}), there exists an ideal $K$ of $F$ contained in $V$, and such that $\GK(F/K)\leq 1$. By Lemma~\ref{sporadics up n down}(1), $L=K^{(d)}$ is an ideal of $F'$ with $\GK(F'/L)\leq 1$. Clearly $L\subseteq V'$ also holds.\par
(2). By \cite[Theorem~7.4(3)]{RSS} a blowup $F$ of $T$ at $\mbf{x}$ exists. The rest of the proof is the same with the relevant definitions in $S$ replaced by those in $T$, and Proposition~\ref{g-div max orders up n down}(2) used instead of Proposition~\ref{g-div max orders up n down}(1).
\end{proof}

Theorem~\ref{RSS 8.11} and Theorem~\ref{main thm converse} answers \cite[Question 9.4]{RSS} with the additional assumption of $\ovl{U}\neq \Bbbk$.  It is worth emphasising that we allow $n$ to be divisible by 3 in Corollary~\ref{g-div max orders up n down generalised} below. Again we remark that $S$ and $T$ are graded differently with $S^{(3d)}=T^{(d)}$ holding.

\begin{cor}\label{g-div max orders up n down generalised}
Let $U$ be a cg graded subalgebra of $S$ with $Q_\gr(U)=Q_\gr(S)^{(d)}$ for some $d\geq 1$ and such that $\ovl{U}\neq \Bbbk$. Let $n\geq 1$ and set $U'=U^{(n)}$.
\begin{enumerate}[(1)]
\item If $U$ is a maximal $S^{(d)}$-order then $U'$ is a maximal $S^{(nd)}$-order.
\item If $U$ is a maximal order then $U'$ is a maximal order.
\end{enumerate}
\end{cor}
\begin{proof}
(1). Suppose first that $d$ is divisible by 3, say $d=3e$. Then by Theorem~\ref{RSS 8.11}(2), $U=V^{(e)}$ where $V=F\cap T$ for some virtual blowup $F$ of $T$. By Theorem~\ref{main thm converse}(2b), $U'=V^{(ne)}$ is a maximal $T^{(ne)}=S^{(nd)}$-order. Now suppose that $d$ is coprime to 3. By Theorem~\ref{RSS 8.11}(1), $U=V^{(d)}$ where $V=F\cap S$ for some virtual blowup $F$ of $S$. If $n$ is coprime to 3, then $U'=V^{(nd)}$ is a maximal $S^{(nd)}$-order by Theorem~\ref{main thm converse}(1b). Finally assume that $n=3m$ for some $m\geq 1$. By Theorem~\ref{3 Veronese of virtual blowup}, $G=F^{(3)}$ is a virtual blowup of $T$. Set $W=G\cap T$. By Theorem~\ref{main thm converse}(2) again, $W^{(md)}=V^{(nd)}=U'$ is a maximal $T^{(md)}=S^{(nd)}$-order.\par

(2). Since $U\subseteq S^{(d)}$, if it is a maximal order it is certainly a maximal $S^{(d)}$-order.
Hence by Theorem~\ref{RSS 8.11}, $U=V^{(e)}$ where $V=F\cap S$ for some virtual blowup $F$ (of either $S$ or $T$) where either $e=d$ or $e=\frac{d}{3}$. In either case $V$ and $F$ are equivalent orders by definition, and hence by Lemma~\ref{equiv orders go up} $U=V^{(e)}$ and $F^{(e)}$ are equivalent orders. But $U$ is a maximal order, so $U=F^{(e)}$. The proof now follows that of part (1) with $F$ in place of $V$, and with Theorem~\ref{main thm converse}(1a) and Theorem~\ref{main thm converse}(2a) in place of Theorem~\ref{main thm converse}(1b) and Theorem~\ref{main thm converse}(2b).
\end{proof}

As immediate corollaries to this classification, we are able to satisfy the intuition that maximal orders are ``nice" rings. A particularly notably example of this is that maximal orders are automatically noetherian.

\begin{cor}\label{RSS 8.11'}
Let $d\geq 1$, and let $U$ be a cg maximal $S^{(d)}$-order such that $\ovl{U}\neq \Bbbk$.  Equivalently, let $F$ be a virtual blowup of $S$ at $\bfx$, $V=F\cap S$ and $U=V^{(d)}$. Then $U$, $V$ and $F$ are all strongly noetherian; in particular noetherian and finitely generated as a $\Bbbk$-algebra.
\end{cor}

\begin{proof}
Since $F$ and $V$ are $g$-divisible cg subalgebras of $S$, they are strongly noetherian by \cite[Proposition~2.9]{RSS} and Remark~\ref{RSS section 2}. By \cite[Proposition~4.9(2)(3)]{ASZ}, $U$ is then strongly noetherian.
\end{proof}

Next we complete the proof of Theorem~\ref{blowup properties}. 

\begin{cor}\label{RSS 8.12}\cite[Corollary~8.12]{RSS}.
Let $d\geq 1$ be coprime to 3, and let $U$ be a cg maximal $S^{(d)}$-order satisfying $\ovl{U}\neq\Bbbk$. Equivalently, let $F$ be a virtual blowup of $S$ at $\bfx$, $V=F\cap S$ and $U=V^{(d)}$. Then $U,V, F$ have cohomological dimension at most 2, they have balanced dualizing complexes and all satisfy the Artin-Zhang $\chi$-conditions.
\end{cor}
\begin{proof}
By Theorem~\ref{RSS 6.7}, $\ovl{V}$ and $\ovl{F}$ are equal in high degrees to a twisted homogeneous coordinate ring. By \cite[Lemma~2.2]{Ro} and \cite[Lemma~8.2(5)]{AZ.ncps}, $\qgr(\ovl{V})$ and $\qgr(\ovl{F})$ has cohomological dimension 1 and, $\ovl{V}$ and $\ovl{F}$ satisfy $\chi$. The fact that $V$ and $F$ satisfy $\chi$ and have cohomological dimension at most 2 follows from \cite[Theorem~8.8]{AZ.ncps}. Since $V$ and $U=V^{(d)}$ are noetherian by Corollary~\ref{RSS 8.11'}, $V$ is a noetherian $U$-module. Then $U$ then also has these properties by \cite[Proposition~8.7(2)]{AZ.ncps}. Finally, $V,F$ and $U$ have balanced dualizing complexes by the above and \cite[Theorem~6.3]{VdB.dualizing}.
\end{proof}

Missing from the above corollaries (compared with Theorem~\ref{S(d) thm}) are the properties Auslander-Gorenstein and Cohen-Macaulay. In \cite[Example~10.4]{RSS} an example of a virtual blowup of $S^{(3)}$ that is neither Auslander-Gorenstein nor Cohen-Macaulay is given. This example can be adapted for a virtual blowup of $S$ to show that these conditions do not hold in general (see \cite[Corollary~7.17]{thesis} for details). The assumption $\ovl{U}\neq\Bbbk $ of Theorem~\ref{RSS 8.11} is annoying yet necessary: \cite[Example~10.8]{RSS} is an example of a maximal $T$-order contained in $gT$. A similar example can be replicated in $S$ and is provided in \cite[Example~7.19]{thesis}. 

\section{A virtual blowup example}

We end the paper with an explicit example of a virtual blowup of $S$. We are able to give algebra generators of a virtual blowup (notoriously a hard problem in noncommutative algebra) as well as realising it as an endomorphism ring.

\begin{theorem}[Proposition~\ref{S(p-p1+p2) main prop}, Remark~\ref{S(p-p1+p2) bad homologically} and Lemma~\ref{F is blowup}] \label{S(p-p1+p2)}
Let $p\in E$ and $\mbf{x}=p-p^\sigma+p^{\sigma^2}$. We set
\begin{itemize}
\item $X_1=S(p+p^{\sigma^2})_1=\{ x\in S_1\,|\; \ovl{x}\in H^0(E,\LL(-\bfx))\};$
\item $X_2=S(p)_1S(p^{\sigma^2})_1\subseteq \{ x\in S_2\,|\; \ovl{x}\in H^0(E,\LL_2(-[\bfx]_2))\};$
\item $X_3=\{ x\in S_3\,|\; \ovl{x}\in H^0(E,\LL_3(-[\bfx]_3))\}.$
\end{itemize}
Put $U=\Bbbk\langle X_1, X_2, X_3\rangle$. Then $U$ is a virtual blowup of $S$ at the virtually effective divisor $\bfx$. In particular
\begin{enumerate}[(1)]
\item $U$ is maximal order contained in $S$.
\item $U$ is noetherian and Corollaries~\ref{RSS 8.11'} and \ref{RSS 8.12} hold.
\item $U$ is neither Auslander-Gorenstein nor AS-Gorenstien nor Cohen-Macaulay. Moreover, $U$ has infinite injective dimension. 
\item Set $R=S(p)$ and $M=R+S(p)_1S_1R$ considered as a right $R$-module. Then $U=\End_{R}(M)$.
\end{enumerate}
\end{theorem}

What's curious about $U$ from Theorem~\ref{S(p-p1+p2)} is the summand $X_2$. It would be more natural to take $U=\Bbbk\langle X_1, X'_2, X_3\rangle$ where $X'_2=\{ u\in S_2\,|\; \ovl{u}\in H^0(E,\LL_2(-[\bfx]_2))\}$, and hence $\ovl{X_2}=B(E,\LL(-\bfx),\sigma)_2$ - a definition more in keeping with the definition of $S(\bfd)$ (Definition~\ref{S(d) def}). Example~\ref{original S(p-p1+p2)} studies this second ring and shows it is far from being a maximal order. This suggests that finding algebra generators for an arbitrary virtual blowup would be more difficult than Theorem~\ref{S(p-p1+p2)} above.\par



Before proceeding with the proof of Theorem~\ref{S(p-p1+p2)}, we require a few computational lemmas to help us. In the coming computations we will be using Notation~\ref{p^sigma^j} and Notation~\ref{[d]_n} frequently, and will be writing $\ovl{S(q)_1}=H^0(E,\LL(-q))$.

\begin{lemma}\label{Ro3.1 applied} Let $B=\ovl{S}=B(E,\LL,\sigma)$. Then
\begin{enumerate}[(1)]
\item For $q,r\in E$, if $r\neq q^{\sigma^2}$ then $\ovl{S(q)_1}\,\ovl{S(r)_1}=H^0(E,\LL_2(-q-r^\sigma))\subseteq B_2$.
\item Let $n\geq 1$ and $p{(0)},p{(1)},\dots, p{(n)}\in E$. Suppose that $p(1)\neq p(0)^{\sigma^2}$. Then
$$ \prod_{i=0}^{n}\ovl{S(p{(i)})_1}= H^0(E,\LL_{n+1}(-p{(0)}-p{(1)}^\sigma-p{(2)}^{\sigma^2}-\dots-p{(n)}^{\sigma^{n}}))\subseteq B_{n+1}.$$
\end{enumerate}
\end{lemma}
\begin{proof}
(1). This is proved in \cite[Lemma~4.1(2)]{Ro}.\par
(2). The $n=1$ case is part (1). We do the case $n=2$, with the general case following by induction. From part (1), we have that $\ovl{S(p{(0)})_1}\,\ovl{S(p{(1)})_1}=H^0(E,\LL_2(-p(0)-p(1)^\sigma))$. Since $\deg\LL_2(-p(0)-p(1)^\sigma)\neq \deg\LL(-p(2))$, \cite[Lemma~3.1]{Ro} implies the multiplication map
\begin{equation*}
H^0(E,\LL_2(-p(0)-p(1)^\sigma))\otimes H^0(E,\LL(-p(2)))\longrightarrow H^0(E,\LL_3(-p(0)-p(1)^\sigma-p(2)^{\sigma^2}))
 \end{equation*}
is surjective. But the image is exactly $\ovl{S(p{(0)})_1}\,\ovl{S(p{(1)})_1}\,\ovl{S(p{(2)})_1}$.
\end{proof}


Another lemma which we will be using regularly in the following calculations is \cite[Lemma~4.1]{Ro}. Below is an extension of it which is implicit in \cite{Ro}.

\begin{lemma}\label{Ro 4.1}
Let $q,r\in E$. Then
\begin{enumerate}[(1)]
\item $S_1S(q)_1=S(q^\sigma)_1S_1$ with $\ovl{S_1S(q)_1}=\ovl{S(q^\sigma)_1S_1}=H^0(E,\LL_2(-q^\sigma))$;
\item $\dim_\Bbbk S(q)_1S(r)_1=4$ if and only if $r\neq q^{\sigma^2}$, while $\dim_\Bbbk S(q)_1S(q^{\sigma^2})_1=3$;
\item $S(q)_1S(r)_1=S(r^{\sigma})_1S(q^{\sigma^{-1}})_1$ if and only if $r\neq q^{\sigma^2},q^{\sigma^{-4}}$; whilst, $S(q)_1S(q^{\sigma^2})_1\subsetneq S(q^{\sigma^3})_1S(q^{\sigma^{-1}})_1.$ \qed
\end{enumerate}
\end{lemma}

Our first lemma on the ring $U$ from Theorem~\ref{S(p-p1+p2)}, specifically the divisor $\bfx$, is a routine computation.

\begin{lemma}\label{bfx v eff}
Let $\bfx=p-p^\sigma+p^{\sigma^2}$, as in Theorem~\ref{S(p-p1+p2)}. Then $\bfx$ is $\sigma$-virtually effective. Moreover, for all $n\geq 2$, $[\bfx]_n$ is effective with $[\bfx]_n=p+p^{\sigma^2}+\dots+p^{\sigma^{n-1}}+p^{\sigma^{n+1}}$. In particular $[\bfx]_2=p+p^{\sigma^3}$ and $[\bfx]_3=p+p^{\sigma^2}+p^{\sigma^4}$ are effective. \qed
\end{lemma}

\begin{notation}\label{p^sigma=p_1} On top of the current notation, we will be using the more compact $p_i=p^{\sigma^i}=\sigma^{-i}(p)$ for $i\in \Z$, for the coming calculations.
\end{notation}

In the next lemma we will be regularly applying Lemma~\ref{Ro 4.1}(3) to two points in the same $\sigma$-orbit. In our new notation this read $S(p_i)_1S(p_j)_1=S(p_{j+1})_1S(p_{i-1})_1\; \text{ if }j\neq i+2,i-4$.

\begin{lemma}\label{ovlU}
Let $U$ and $\bfx$ be as in Theorem~\ref{S(p-p1+p2)}. Set $B=B(E,\LL(-\bfx),\sigma)$. Then
$$\ovl{U}=\Bbbk\oplus\ovl{X_1}\oplus \ovl{X_2}\oplus B_{\geq 3}.$$
\end{lemma}
\begin{proof}
Write $V=\ovl{U}$. By definition $V=\Bbbk\langle \ovl{X_1},\ovl{X_2},\ovl{X_3}\rangle$, and clearly $V_{\leq 1}=\Bbbk\oplus \ovl{X_1}$. Since $p\neq p_2$, it is easy to prove $\ovl{X_1}=\ovl{S(p)_1}\cap \ovl{S(p_2)_1}$. Then
\begin{equation}\label{ovlU1^2} \ovl{X_1}^2\subseteq \ovl{S(p)_1}\ovl{S(p_2)_1}=\ovl{X_2},\end{equation}
which shows $V_2=\ovl{X_2}$. By Lemma~\ref{bfx v eff}, $[\bfx]_3$ is effective, and so $H^0(E,\LL_3(-[\bfx]_3))\subseteq \ovl{S}_3$. It is then obvious from the definition that $\ovl{X_3}=B_3$. Since $\ovl{X_1},\ovl{X_2}\subseteq B$ we have
 \begin{equation}\label{ovlU1U2+ovlU2U1} \ovl{X_1X_2}+\ovl{X_2X_1}\subseteq B_3=\ovl{X_3}.\end{equation}
It follows $V_3=\ovl{X_3}.$ Since $V$ is generated in degrees less than or equal to 3, we have also proved $V\subseteq B$.\par

Now using Lemma~\ref{Ro 4.1}(3) for equality 1 below, Lemma~\ref{Ro3.1 applied}(2) for 2, and Lemma~\ref{bfx v eff} for 3, we have
\begin{multline}\label{V4}
V_4\supseteq \ovl{X_2}^2=\ovl{S(p)_1}(\ovl{S(p_2)_1}\ovl{S(p)_1})\ovl{S(p_2)_1}\,\overset{1}=\,
\ovl{S(p)_1}(\ovl{S(p_1)_1}\ovl{S(p_1)_1})\ovl{S(p_2)_1}\\ \overset{2}=H^0(E,\LL(-p-p_2-p_3-p_5))\,\overset{3}=\,H^0(E,\LL(-[\bfx]_4))=B_4.
\end{multline}
Also, using Lemma~\ref{Ro3.1 applied}(2), we can write
\begin{equation}\label{V3} \ovl{X_3}=\ovl{S(p)_1}\ovl{S(p_1)_1}\ovl{S(p_2)_1}.\end{equation}
So a similar calculation using Lemma~\ref{Ro 4.1}(3), Lemma~\ref{Ro3.1 applied}(2) and Lemma~\ref{bfx v eff} obtains
\begin{multline}\label{V5} V_5\supseteq \ovl{X_2}\ovl{X_3}= (\ovl{S(p)_1}\ovl{S(p_2)_1})(\ovl{S(p)_1}\ovl{S(p_1)_1}\ovl{S(p_2)_1}) =\ovl{S(p)_1}(\ovl{S(p_1)_1}\ovl{S(p_1)_1})\ovl{S(p_1)_1}\ovl{S(p_2)_1}\\
=H^0(E,\LL_5(-p_0-p_2-p_3-p_4-p_6))=H^0(E,\LL_5(-[\bfx]_5))=B_5\end{multline}
Now for $n\geq 6$, it is easily seen that we can write $n=3a+4b+5c$ with $a,b,c$ nonnegative integers. In which case
$$V_n\supseteq V_3^{a}V_4^{b}V_5^{c}=B_3^aB_4^bB_5^c=B_{3a+4b+5c}=B_n$$
by equations (\ref{V4}), (\ref{V3}) and (\ref{V5}), Lemma~\ref{Ro3.1 applied}(2) and Lemma~\ref{bfx v eff}.
\end{proof}

At this point we note that (\ref{ovlU1^2}) can be lifted to $S$ since $gS\subseteq S_{\geq 3}$. Similarly, as $X_3\cap gS=g\Bbbk= S_3\cap gS$, (\ref{ovlU1U2+ovlU2U1}) can also be lifted to $S$. This proves that
\begin{equation}\label{X_i=U_i}
X_1,\, X_2,\text{ and }X_3\text{ are the homogeneous elements of }U\text{ of degree 1, 2, and 3}.
\end{equation}

Next we look to realise $U$ as an endomorphism ring.

\begin{notation}\label{virtual blowup notation}
Let $R=S(p)$, $M=R+S(p)_1S_1R$ and $H=\End_{R}(M)$. We also put $F=\End_R(M^{**})$ and $V=F\cap S$.
\end{notation}

\begin{lemma}\label{U subset H}
Retain the notation of Theorem~\ref{S(p-p1+p2)} and Notation~\ref{virtual blowup notation}. Then
$U\subseteq H.$
\end{lemma}
\begin{proof}
Here we again use Notation~\ref{p^sigma=p_1}. We warn that because $gS\subseteq S_{\geq 3}$, we can and will regularly identify $S_{\leq 2}=\ovl{S_{\leq 2}}$. Since $U$ is generated by $X_1,X_2,X_3$ we need to show $X_1,X_2,X_3\subseteq H$.  As $H=\{x\in Q_\gr(S)\,|\; xM\subseteq M\}$, we must prove $X_iM\subseteq M$ for each $i$. \par

\textit{Proof of} $X_1M\subseteq M$. Since $X_1\subseteq S(p)_1=R_1$, we clearly have $X_1R\subseteq R\subseteq M$. So it is
\begin{equation}\label{U1M subset M}   X_1S(p)_1S_1R\subseteq M\end{equation}
that we need to prove. First we show that $\ovl{X_1}\,\ovl{S(p)_1S_1}\subseteq \ovl{R_3}=H^0(E,\LL_3(-p-p_1-p_2))$. For this, we have $\ovl{X_1}=H^0(E,\LL(-p-p_2))$, and $\ovl{S(p)_1S_1}=H^0(E,\LL_2(-p))$ by Lemma~\ref{Ro 4.1}(1). Thus $\ovl{X_1}\,\ovl{S(p)_1S_1}$ is the image of the natural map $H^0(E,\LL(-p-p_2))\otimes H^0(E,\LL_2(-p)) \longrightarrow H^0(E,\LL(-p-p_2)\otimes \LL_2(-p)^\sigma)$. Further, $\LL(-p-p_2)\otimes \LL_2(-p)^\sigma=\LL_3(-p-p_1-p_2)$. Hence indeed $\ovl{X_1}\,\ovl{S(p)_1S_1}\subseteq \ovl{R_3}$. This implies $X_1S(p)_1S_1\subseteq R_3+(S_3\cap Sg)$. But $\deg(g)=3$, and so $S_3\cap Sg=S_0g=\Bbbk g=R_0g=R_3\cap Sg$.
Hence $X_1S(p)_1S_1\subseteq R_3$ and (\ref{U1M subset M}) follows. \par

\textit{Proof of} $X_2M\subseteq M$. First, clearly
\begin{equation}\label{U2M subset H 1} X_2R=S(p)_1S(p_2)_1R\subseteq S(p)_1S_1R\subseteq M.\end{equation}
Next, using Lemma~\ref{Ro 4.1}(1) for equality 1 and Lemma~\ref{Ro 4.1}(3) for equality 2, we also have
\begin{multline}\label{U2M subset H 2} X_2S(p)_1S_1R=(S(p)_1S(p_2)_1)(S(p)_1S_1R)\overset{1}=(S(p)_1S_1)(S(p_1)S(p_{-1}))R\\
\overset{2}=(S(p)_1S_1)S(p)_1^2R=(S(p)_1S_1)R_1^2R\subseteq S(p)_1S_1R\subseteq M. \end{multline}
Equations (\ref{U2M subset H 1}) and (\ref{U2M subset H 2}) together show $X_2M=X_2(R+S(p)_1S_1R)\subseteq M$.\par

\textit{Proof of} $X_3M\subseteq M$. From (\ref{V3}) we have $\ovl{X_3}=\ovl{S(p)_1}\ovl{S(p_1)_1}\ovl{S(p_2)_1}$. Since $\deg(g)=3$ and $gS\cap S_3=\Bbbk g$, this implies $X_3=S(p)_1S(p_1)_1S(p_2)_1+\Bbbk g$ (in fact one can prove $g\in S(p)_1S(p_1)_1S(p_2)_1$ but this is not necessary for us). As $g\in R$ and is central, we clearly we have $gM=Mg\subseteq M$. Therefore what we need to prove is
\begin{equation}\label{U3 in H} S(p)_1S(p_1)_1S(p_2)_1M=S(p)_1S(p_1)_1S(p_2)_1(R+S(p)_1S_1R)\subseteq M.\end{equation}
First, using Lemma~\ref{Ro 4.1}(1) for equality 1 below, we have
\begin{multline}\label{U3M subset H 1}
(S(p)_1S(p_1)_1S(p_2)_1)R\subseteq (S(p)_1S(p_1)_1S_1)R\overset{1}=(S(p)_1S_1S(p)_1)R\\
=(S(p)_1S_1)R_1R \subseteq S(p)_1S_1R\subseteq M.
\end{multline}
Secondly, with multiple uses of  Lemma~\ref{Ro 4.1},
\begin{multline}\label{U3M subset H 2}(S(p)_1S(p_1)_1S(p_2)_1)(S(p)_1S_1)R=(S(p)_1S_1)(S(p)_1S(p_1)_1S(p_{-1}))R \\
=(S(p)_1S_1)S(p)_1^3R=(S(p)_1S_1)R_3R\subseteq(S(p)_1S_1)R\subseteq M.
\end{multline}
Equations (\ref{U3M subset H 1}) and (\ref{U3M subset H 2}) together give (\ref{U3 in H}).
\end{proof}

\begin{lemma}\label{M H with bars}
Retain the notation of Theorem~\ref{S(p-p1+p2)} and Notation~\ref{virtual blowup notation}. Then
\begin{enumerate}[(1)]
\item $\ovl{M}\ehd \bigoplus_{n\geq 0} H^0(E,\OO_E(p^\sigma)\otimes \LL(-p)_n)=\bigoplus_{n\geq 0} H^0(E,\LL_n(p-p^{\sigma^2}-\dots-p^{\sigma^{n-1}}));$
\item $\ovl{H}\ehd B(E,\LL(-\bfx),\sigma)\ehd\ovl{U}$.
\end{enumerate}
\end{lemma}
\begin{proof}
(1). Let $n\gg 0$. As $R$ is generated in degree 1, $R_n\subseteq R_1S_1R_{n-2}=S(p)_1S_1R_{n-2}$, and so $\ovl{M}_n=\ovl{S(p)_1}\ovl{S_1}\ovl{S(p)_{n-2}}$. By Lemma~\ref{Ro 4.1}(1), $\ovl{S(p)_1}\ovl{S_1}=H^0(E,\LL_2(-p))$. Hence $\ovl{M}_n$ is the image of the map $H^0(E,\LL_2(-p))\otimes H^0(E,\LL(-p)_{n-2})\overset{\mu}\longrightarrow H^0(E,\LL_2(-p)\otimes\LL(-p)_{n-2}^{\sigma^2})$.
On the right hand side $\LL_2(-p)\otimes\LL(-p)_{n-2}^{\sigma^2}=\LL_n(-p-p^{\sigma^2}-\dots-p^{\sigma^{n-1}})= \OO_E(p^\sigma)\otimes \LL(-p)_n$.
So provided $\deg( \OO_E(p^\sigma)\otimes \LL(-p)_n)=2n+1>2$, the map $\mu$ is surjective by \cite[Lemma~3.1]{Ro}.\par
(2). By Lemma~\ref{U subset H}, \cite[Lemma~2.12(3)]{RSS} and Remark~\ref{RSS section 2} , $\ovl{U}\subseteq \ovl{H}\subseteq \End_{\ovl{R}}(\ovl{M})$. On the other hand $\End_{\ovl{R}}(\ovl{M})\ehd B(E,\LL(-\bfx),\sigma)$ by part (1) and \cite[Lemma~6.14(1)]{RSS}. Then by Lemma~\ref{ovlU}, $\ovl{U}\ehd B(E,\LL(-\bfx),\sigma)\ehd \ovl{H}$.
\end{proof}

Now we study the 3-Veronese of $U$. For the next proposition we must recall Notation~\ref{3 Veronese notation2} and Definition~\ref{T(d) def} for the blowup subalgebras $T(\bfd)$ of $T$.

\begin{prop}\label{3 Veronese of H}
Retain the notation of Theorem~\ref{S(p-p1+p2)} and Notation~\ref{virtual blowup notation}. Then
$$U^{(3)}=H^{(3)}=T(p+p^{\sigma^2}+p^{\sigma^4}),$$
where $T(p+p^{\sigma^2}+p^{\sigma^4})$ is the blowup of $T$ from Definition~\ref{T(d) def}. In particular, $U^{(3)}$ is a maximal order.
\end{prop}
\begin{proof}
Let $\bfy=p+p^{\sigma^2}+p^{\sigma^4}$. By Lemma~\ref{bfx v eff}, $\bfy=[\bfx]_3$. So by definition, $T(\bfy)=\Bbbk\langle X_3\rangle$. In particular, with Lemma~\ref{U subset H}, we already have $T(\bfy)\subseteq U^{(3)}\subseteq H^{(3)}$. It is hence enough to prove $H^{(3)}= T(\bfy)$. We do this by showing they are equivalent orders. \par
Since $M_R$ is finitely generated we can find a nonzero $x\in R$ such that $xM\subseteq R$. Choose any nonzero $y\in M$, then clearly $xHy\subseteq xM\subseteq R$. Now notice $(Mx)M\subseteq MR\subseteq M$, meaning $Mx\subseteq H$ also. As $yR\subseteq M$, we then have $yRx\subseteq H$. So $H$ and $R$ are equivalent orders. By Lemma~\ref{equiv orders go up}, $H^{(3)}$ and $R^{(3)}$ are equivalent orders. Now $R^{(3)}=T(p+p^{\sigma}+p^{\sigma^2})$ by Theorem~\ref{S(d) thm}(1) and is an equivalent order with $T(p+p^{\sigma^2}+p^{\sigma^4})=T(\bfy)$ by \cite[Corollary~5.27]{RSS}. Thus $H^{(3)}$ and $T(\bfy)$ are equivalent orders. However $T(\bfy)$ is a maximal order by \cite[Theorem~1.1]{Ro}, so $H^{(3)}=T(\bfy)$.
\end{proof}

\begin{lemma}\label{F is blowup}
Let $F$, $V$ and $\bfx$ be as in Notation~\ref{virtual blowup notation} and Theorem~\ref{S(p-p1+p2)}. Then $F$ is a blowup of $S$ at $\bfx$. In other words,
\begin{enumerate}[(1)]
\item $(F,F\cap S)$ is a maximal order pair of $S$ in the sense of Definition~\ref{max order pair}.
\item $\ovl{F}\ehd B(E,\LL(-\bfx),\sigma)$.
\end{enumerate}
Moreover $U\subseteq H\subseteq \widehat{H}\subseteq F$ holds.
\end{lemma}
\begin{proof}
By Lemma~\ref{hat end commute}, $\widehat{H}=\End_R(\widehat{M})$, while $M^{**}=(\widehat{M})^{**}$ by \cite[Lemma~2.13(3)]{RSS} and Remark~\ref{RSS section 2}. By Proposition~\ref{RSS 6.4}, $(F,F\cap S)$ is then the unique maximal order pair containing and equivalent to $\widehat{H}$. This and Lemma~\ref{U subset H} show $U\subseteq H\subseteq \widehat{H}\subseteq F$. By Proposition~\ref{RSS 6.7}, $\ovl{F}\ehd \ovl{\widehat{H}}\ehd B(E,\LL(-\bfx'),\sigma)$ for some $\sigma$-virtually effective $\bfx'$. We must prove $\bfx=\bfx'$. Now by Proposition~\ref{C to C hat}, $H$ is an  equivalent order with $\widehat{H}$, and hence $H$ is contained in and equivalent to $F$ as well. By Lemma~\ref{equiv orders go up}, $H^{(3)}$ is then contained in and equivalent to $F^{(3)}$. But $H^{(3)}=T(p+p^{\sigma^2}+p^{\sigma^4})$ by Proposition~\ref{3 Veronese of H}; and furthermore $p+p^{\sigma^2}+p^{\sigma^4}=[\bfx]_3$ by Lemma~\ref{bfx v eff}. Now $T([\bfx]_3)$ is a maximal order by \cite[Theorem~1.1]{Ro}, and thus $F^{(3)}=T([\bfx]_3)$. In particular $\ovl{F}^{(3)}=\ovl{T([\bfx]_3)}=B(E,\LL_3(-[\bfx]_3),\sigma)$. On the other hand $\ovl{F}^{(3)}\ehd B(E,\LL(-\bfx'),\sigma)^{(3)}$; therefore $ H^0(E,\LL_{3n}(-[\bfx]_{3n}))= B(E,\LL(-\bfx),\sigma)_{3n}= B(E,\LL(-\bfx'),\sigma)_{3n}=H^0(E,\LL_{3n}(-[\bfx']_{3n}))$ for $n\gg 0$.
Since for $n\gg0$, $\LL_{3n}(-[\bfx]_{3n})$ and $\LL_{3n}(-[\bfx']_{3n})$ are generated by their global sections, it follows that $[\bfx]_{3n}=[\bfx']_{3n}$. From this it follows $\mbf{x}=\mbf{x'}$.
\end{proof}

We now look to improve $U\subseteq H\subseteq \widehat{H}\subseteq F$ to equalities. This is first achieved in $\ovl{S}=B(E,\LL,\sigma)$.

\begin{lemma}\label{S(p-p1+p2) with bars}
Retain the notation  from Theorem~\ref{S(p-p1+p2)} and Notation~\ref{virtual blowup notation}. Then
$\ovl{U}=\ovl{H}=\ovl{\widehat{H}}=\ovl{F}.$
\end{lemma}

\begin{proof}
Set $B=B(E,\LL(-\bfx),\sigma)$ considered as a subalgebra of $Q_\gr(\ovl{S})=\Bbbk(E)[t,t^{-1};\sigma]$. By Lemma~\ref{F is blowup}, $\ovl{F}\ehd B$. Since $B$ is Auslander-Gorenstein and Cohen-Macaulay by \cite[Lemma~2.2(3)]{Ro}, we have that $\ovl{F}\subseteq B$. Thus with Lemma~\ref{F is blowup}, $\ovl{U}\subseteq \ovl{H}\subseteq \ovl{\widehat{H}}\subseteq \ovl{F}\subseteq B$. So proving $\ovl{U}=\ovl{F}$ is enough. By Lemma~\ref{ovlU}, $\ovl{U}_{\geq 3}=\ovl{F}_{\geq 3}$. \par
To prove $\ovl{U}_2=\ovl{F}_2$ assume otherwise that $\ovl{U}_2\subsetneq \ovl{F}_2$. Because $gS\subseteq S_{\geq 3}$, we can and will identify $S_{\leq 2}=\ovl{S_{\leq 2}}$. Also we note that since $[\bfx]_2=p+p^{\sigma^3}$ is effective, $B_2\subseteq \ovl{S}_2$. Now, by the Riemann-Roch Theorem $\dim_\Bbbk B_2=4$, whilst $\dim_\Bbbk \ovl{U}_2=3$ by Lemma~\ref{Ro 4.1} and (\ref{X_i=U_i}). Thus it must be the case that $\ovl{F_2}=B_2$. In other words, $F_2=\{ x\in S_2\,|\; \ovl{x}\in B_2\}$. In which case $F$ contains the ring $U'=\Bbbk\langle X_1, F_2, X_3\rangle$. But this is exactly $U'$ from Example~\ref{original S(p-p1+p2)}. Since $F$ is $g$-divisible we would then have $F\supseteq \widehat{U'}$. However, in Example~\ref{original S(p-p1+p2)}(2) we will show that $\widehat{U'}=S$. This would mean $F\supseteq S$, which certainly contradicts Lemma~\ref{F is blowup}(2). Thus we must have $F_2=U_2$, or equivalently $\ovl{F}_2=\ovl{U}_2$. \par
Lastly we show $\ovl{U}_1=\ovl{F}_1$. Again, since $\dim_\Bbbk B_1=2$ and $\dim_\Bbbk\ovl{U}_1=1$, if $\ovl{U}_1\neq \ovl{F}_1$, then $\ovl{F}_1=B_1$. But in this case, $\ovl{F}_2\supseteq B_1^2=B_2$ by \cite[Lemma~3.1]{RSS}. This would contradict $\ovl{U}_2=\ovl{F}_2\subsetneq B_2$, hence $\ovl{U}_1=\ovl{F}_1$.
\end{proof}

Finally we can conclude $U$ is a virtual blowup.

\begin{prop}\label{S(p-p1+p2) main prop}
Let $U$ and $\bfx$ be as in Theorem~\ref{S(p-p1+p2)}. Then $U$ is a virtual blowup of $S$ at $\bfx$. More specifically, retaining notation from Notation \ref{virtual blowup notation}, $U=H=\widehat{H}=F$.
\end{prop}
\begin{proof}
By Lemma~\ref{F is blowup}, $U\subseteq H\subseteq\widehat{H}\subseteq F$. So we need to show $U=F$. We first prove $U$ is $g$-divisible. Set $W=\widehat{U}$. Since $U\subseteq F$ and $F$ is $g$-divisible by Lemma~\ref{F is blowup}, $U\subseteq W\subseteq F$. Therefore $\ovl{U}=\ovl{W}=\ovl{F}$ by Lemma~\ref{S(p-p1+p2) with bars}. As $gS\subseteq S_{\geq 3}$, this shows $U_{\leq 2}=W_{\leq 2}$. Proceeding by induction, let $n\geq3$ and assume $W_{<n}=U_{<n}$. Let $w\in W_n$, say with $wg^k\in U$. Since $\ovl{U}=\ovl{W}$, certainly $W\subseteq U+gS$, and so there exists $u\in U_n$ and $s\in S_{n-3}$ such that $w=u+sg$. We then have $sg^{k+1}+ug^k=wg^k\in U$, and so $sg^{k+1}=wg^k-ug^k \in U$ also. Hence $s\in W$. Moreover, since $\deg(w)=\deg(sg)$, we have $\deg(s)<\deg(w)=n$. By induction, $s\in U$. Therefore $w=u+gs\in U$ also, which proves $U=W$. \par

Now we have two $g$-divisible rings $U\subseteq F$, with $\ovl{U}=\ovl{F}$. Comparing Hilbert series we have $h_U(t)=h_{\ovl{U}}(t)/(1-t^3)=h_{\ovl{F}}(t)/(1-t^3)=h_F(t)$. Since $U\subseteq F$, this forces $U=F$.
\end{proof}

As mentioned after Theorem~\ref{S(p-p1+p2)}, the definition of $X_2$ is surprising. We now study a ring, that initially (at least to the author) seemed a more natural definition for the virtual blowup $U$. We also note that Example~\ref{original S(p-p1+p2)} completes the missing step in Lemma~\ref{S(p-p1+p2) with bars}.

\begin{example}\label{original S(p-p1+p2)} Let $U$ and $\bfx$ be as in Theorem~\ref{S(p-p1+p2)}.
Set $U'=\Bbbk\langle X'_1,X'_2,X'_3\rangle$, where this time $X'_i=\{x\in S_i\,|\;\ovl{x}\in H^0(E,\LL_i(-[\bfx]_i))\}$ for $i=1,2,3$. Then
\begin{enumerate}[(1)]
\item $U\subseteq U'$;
\item $\widehat{U'}=S$;
\item $U'$ is neither left nor right noetherian.
\end{enumerate}
\end{example}
\begin{proof}
(1). This is obvious from the definitions of $U$ and $U'$.\par
(2). Here we use Notation~\ref{p^sigma=p_1}. Identifying $\ovl{S_{\leq 2}}$ and $S_{\leq 2}$; $X'_2=S(p_3)_1S(p_{-1})_1$ by Lemma~\ref{Ro3.1 applied}(1). By the Riemann-Roch Theorem $\dim_\Bbbk X'_2 = 4$. To start we prove that $g\in S(p_3)_1S(p_{-1})_1S(p_3)_1$.
Let $Y=S(p_{-1})_1S(p_3)_1$. Write $S(p_3)_1=v_1\Bbbk+v_2\Bbbk$. Then $S(p_3)_1Y=v_1Y+v_2Y$, and therefore
\begin{equation}\label{Ro 4.6 trick}
 \dim_\Bbbk S(p_3)_1Y=2\dim_\Bbbk Y-\dim_\Bbbk(v_1Y\cap v_2Y).
\end{equation}
Now, applying \cite[Lemma~4.2]{Ro} with $q=p_3=\sigma^{-3}(p)$, we can choose a basis of $S(p_5)_1=w_1\Bbbk+w_2\Bbbk$ such that $v_1w_1+v_2w_2=0$ and $v_1S\cap v_2S=v_1w_1S=v_2w_2S$. In which case
\begin{equation}\label{Ro 4.6 Z}
v_1Y\cap v_2Y=\{v_1w_1s\,|\; w_1s,w_2s\in Y\}=v_1w_1Z,\; \text{ where } Z=\{s\in S_1\;|\,S(p_5)_1s\subseteq Y\}.
\end{equation}
In particular $\dim_\Bbbk(v_1Y\cap v_2Y)=\dim_\Bbbk Z$. We compute $\ovl{Z}$ which, as $gS\subseteq S_{\geq 3}$, can be identified with $Z$. Take a section $\ovl{s}\in\ovl{S}_1=H^0(E,\LL)$, say $\ovl{s}$ vanishes at the effective divisor $\bfd$. Then $\ovl{S(p_5)}_1\ovl{s}$ will consist of global sections of $\LL_2$ vanishing at $p_5+\bfd^\sigma$. On the other hand $\ovl{Y}$ consists of global section of $\LL_2$ vanishing at $p_{-1}+p_4$ by Lemma~\ref{Ro3.1 applied}(1). In which case $\ovl{s}\in Z$ if and only if $\ovl{S(p_5)_1}\ovl{s}$ consists of global sections of $\LL_2$ vanishing at $p_{-1}+p_4$; that is, if and only if $p_5+\bfd^\sigma\geq p_{-1}+p_4$; or equivalently $p_{-2}+p_3\leq \bfd$. Thus $\ovl{Z}=H^0(E,\LL(-p_{-2}-p_3))$. In particular, $\dim_\Bbbk Z=\dim_\Bbbk\ovl{Z}=1$ by the Riemann-Roch Theorem. This means $\dim_\Bbbk S(p_3)_1Y=7$ by (\ref{Ro 4.6 trick}) and Lemma~\ref{Ro 4.1}. On the other hand $\ovl{S(p_3)_1Y}=H^0(E,\LL_3(-p-p_3-p_5))$, which is 6 dimensional by the Riemann-Roch Theorem. We hence must have $g\in S(p_3)Y=S(p_3)_1S(p_{-1})_1S(p_3)_1$ as claimed.

A similar argument shows that $g\in S(p_{-1})_1S(p_3)_1S(p_{-1})_1$. It follows that $(X'_2)^2\supseteq gS(p_{-1})_1$ and $gS(p_3)_1\subseteq (X'_2)^2$. Hence $(X'_2)^2\supseteq g(S(p_{-1})_1+S(p_3))_1=gS_1$. This shows $S_1\subseteq \widehat{U'}$. But $S=\Bbbk\langle S_1\rangle$; therefore $\widehat{U'}=S$. \par

(3). By (2) and Proposition~\ref{C to C hat}, there exists an $k\geq 0$ such that $g^kS\subseteq U'$. Hence, if $U'$ is right noetherian, then $S_{U'}$ would be finitely generated. In which case, $\ovl{S}_{\ovl{U'}}$ would also be finitely generated. Put $\NN=\LL(-\bfx)$, by part (1) and Lemma~\ref{ovlU} we have $\ovl{U}\subseteq\ovl{U'}\subseteq B(E,\NN,\sigma)\ehd \ovl{U}.$ Hence also $\ovl{U'}\ehd B(E,\NN,\sigma)$. Thus by \cite[Theorem 1.3]{AV}, $\ovl{S}_n=H^0(E,\OO_E(\mbf{u})\otimes \NN_n)$ for some divisor $\mbf{u}$ and for all $n\gg 0$. It would then follow from the Riemann-Roch Theorem that $\dim_\Bbbk \ovl{S}_n=n+\deg(\mbf{u})$. However $\ovl{S}_n=H^0(E,\LL_n)$, and hence $\dim_\Bbbk\ovl{S}_n=3n$ for all $n$. This would give a contradiction, and therefore $U'$ cannot be right noetherian. A symmetric argument shows $U'$ is not left noetherian.
\end{proof}

\begin{remark}\label{S(p-p1+p2) bad homologically}
The final piece of Theorem~\ref{S(p-p1+p2)} is to show that we cannot expect nice homological properties to hold for virtual blowups. Theorem~\ref{S(p-p1+p2)}(3) can be proved in the exact manner as \cite[Example~10.4 and Remark~10.7]{RSS}.   
\end{remark}

\newpage

\bibliographystyle{plain}
\bibliography{Bib}

\end{document}